\documentclass[a4paper,12pt]{article}
\usepackage{amsrefs,amstext,amsmath,amsthm,amsfonts,amssymb,enumerate,amscd,makeidx,mathrsfs}
\usepackage[refpage,intoc]{nomencl}
\usepackage{hyperref}
\usepackage[dvipdf]{graphicx}
\usepackage[margin=3 true cm]{geometry}
\usepackage{todonotes}
\usepackage{xspace}

\newtheorem{theorem}{Theorem}[section]
\newtheorem{proposition}[theorem]{Proposition}
\newtheorem{lemma}[theorem]{Lemma}
\newtheorem{corollary}[theorem]{Corollary}

\newtheorem{definition}[theorem]{Definition}
\newtheorem{remark}[theorem]{Remark}

\makenomenclature

\title{Herz-Schur Multipliers of Fell Bundles}
\author{Weijiao He}

\begin{document}

\maketitle

\begin{abstract}
In this paper we develop the notion of Herz-Schur multipliers to the context of Fell bundle, and we give a necessary and sufficient condition if the reduced cross-sectional $C^{\ast}$-algebra $C_r^{\ast}(\mathfrak{B})$ of Fell bundle $\mathfrak{B}$ over discrete group $G$ is nuclear in terms of this generalized notion. As an application, we prove that if $C_r^{\ast}(\mathfrak{B})$ is nuclear, then for any subgroup $H \subset G$, the $C^{\ast}$-algebra $C_r^{\ast}(\mathfrak{B}_H)$ is nuclear.
\end{abstract}

\section{Introduction}

The notion of a Schur multiplier has its origins in the work of I. Schur in the early $\rm{20^{th}}$ century, and is based on the entry-wise (or Hadamard) product of matrices. More specifically, a bounded function $\varphi: \mathbb{N} \times \mathbb{N} \to \mathbb{C}$ is called a Schur multiplier if $(\varphi(i,j)a_{i,j})$ is the matrix of a bounded linear operator on $\ell^2$ whenever $(a_{i,j})$ is such. Hence to a Schur multiplier $\varphi: \mathbb{N} \times \mathbb{N} \to \mathbb{C}$ we can associate an operator $S_{\varphi}$ on $\mathcal{B}(\ell^2)$. Based on a concrete description of Schur multipliers which was given by A. Grothenieck in \cite{Classic}, Schur multipliers can be identified with $\ell^{\infty} \otimes_{eh} \ell^{\infty}$, the extended Haagerup tensor product of two copies of $\ell^{\infty}$.

Among the large number of applications of Schur multipliers is the description
of the space $MA^{cb}(G)$ of completely bounded multipliers, also
known as Herz-Schur multipliers of the Fourier algebra $A(G)$ of a locally
compact group G, introduced by J. de Canniere and U. Haagerup in \cite{M017}.
Namely, as shown by M. Bozejko and G. Fendler \cite{M010}, $M^{cb}A(G)$ can be
isometrically identified with the space of all Schur multipliers on $G \times G$ of
Toeplitz type. Furthermore,  in connection with Schur multiplier, Herz-Schur multiplier has very important application in the study of the nuclearity of the group $C^{\ast}$-algebra, i.e for a discrete group $G$, the reduced group $C^{\ast}$-algebra $C_r^{\ast}(G)$ is nuclear if and only if there is a net of completely positive Herz-Schur multiplier $\varphi_i: G \to \mathbb{C}$ such that $\varphi_i(x) \to 1$ for all $x \in G$ (see e.g \cite{approximation}). 

Recently, in \cite{MTT16}, McKee, Todorov and Turowska generalized the notion of Schur multipliers and Herz-Schur multipliers to the $C^{\ast}$-algebra valued case: a class of \emph{Schur A-multipliers} and a class of \emph{Herz-Schur multiplier of semi-direct product bundle} are identified, where $A$ is a $C^{\ast}$-algebra faithfully represented on a Hilbert space $H$. In this `operator-valued' case, the starting point is a function $\phi$ defined on the direct product $X \times Y$, where $X$ and $Y$ are standard measure space, and taking values in the space $CB(A, \mathcal{B}(H))$ of all completely bounded maps from $A$ into the $C^{\ast}$-algebra $\mathcal{O}(H)$ of all bounded linear operators on $H$. The associated operator $S_{\phi}$ acts from $\mathcal{O}_c(L_2(X), L_2(Y)) \otimes A$ into $\mathcal{O}_c(L_2(X), L_2(Y)) \otimes \mathcal{O}(H)$; the function $\phi$ is called a \emph{Schur A-multiplier} if the map $S_{\phi}$ is completely bounded. Similarly, for semi-direct bundle $A \underset{\tau}{\times} G$, we start from a function $\varphi$ from $G$ into $CB(A,A)$, the associated operator $S_{\varphi}$ acts from $\mathfrak{L}_2(A \underset{\tau}{\times}G)$ (the square-integrable cross sections of $A \underset{\tau}{\times}G$) to itself; the map $\varphi$ is called Herz-Schur multiplier of $A \underset{\tau}{\times}G$ if $S_{\varphi}$ is completely bounded with respect to the norm of $C_r^{\ast}(A \underset{\tau}{\times}G)$. In \cite{MR3860401}, McKee, Skalski, Todorov and Turowska use these generalized notions to generalize the classical result we mentioned in the last paragraph: for semi-direct product bundle $A \underset{\tau}{\times}G$ over discrete group $G$, $C_r^{\ast}(A \underset{\tau}{\times}G)$ is nuclear $C^{\ast}$-algebra if and only if there is a net completely positive Herz-Schur multiplier of $A \underset{\tau}{\times}G$ such that $\varphi_i(x)(a) \to a$ for all $x \in G$ and $a \in A$.

In this paper, we will generalize the notion of Herz-Schur multiplier to Fell bundle $\mathfrak{B}$ over locally compact group by borrowing the ideas of \cite{MTT16} and \cite{MR3860401}, then we will give a necessary and sufficient condition of that $C_r^{\ast}(\mathfrak{B})$ is nuclear $C^{\ast}$-algebra in terms of multipliers when $G$ is a discrete group, finally as an application we prove: If $\mathfrak{B}$ is a Fell bundle over discrete group $G$ such that $C_r^{\ast}(\mathfrak{B})$ is nuclear, then for any subgroup $H \subset G$ the $C^{\ast}$-algebra $C_r^{\ast}(\mathfrak{B})$ is nuclear as well.

The plan of this paper is: In the Section 2, we give some notations and conventions which we will use in this paper; in the Section 3, we define Schur multipliers of Fell bundle; in Section 4, we give a characterization of Schur multipliers of Fell bundle which is analogous to \cite[Theorem 2.6]{MTT16}, and during this process we include the non-unital version Stinpring's Theorem in a proposition which will be very important in Section 6 but with an easier proof; in Section 5 we study the generalization of Herz-Schur multipliers in the context of Fell bundles, including the generalized transference theorem between Schur multipliers and Herz-Schur multipliers (see e.g \cite[Theorem 3.8]{MTT16}); finally in Section 6 we study the problem concerning the nuclearity of the reduced $C^{\ast}$-algebra by aid of the notion of the generalized Herz-Schur multipliers.

\section{Preliminary}

We begin in this section with some conventions which will be assumed without reference. First of all, we refer the reader to \cite[II.13]{MR936628} for the notion of Banach bundle, \cite[VIII.2, VIII.3, VIII.16]{MR936629} for the notion of $C^{\ast}$-algebraic bundle, to \cite[VIII.9]{MR936629} for the basic knowledge about the representation theory of $C^{\ast}$-algebraic bundle. As usual, we call $C^{\ast}$-algebraic bundle Fell bundle. 

For Banach bundle $\mathfrak{D}$ over locally compact Hausdorff space $M$ with a fixed Borel measure, we use symbol $\mathfrak{L}(\mathfrak{D})$ to denote the space of continuous cross-section vanishing outside some compact subset of $M$, use symbol $\mathfrak{L}_p(\mathfrak{D})$ to denote the space of $p$-integrable cross-section of $\mathfrak{B}$ (see \cite[II.15]{MR936628}).

In this paper, we assume that $G$ is a fixed group which is either discrete with counting measure $\mu$ or is locally compact and second countable with fixed Haar measure $\mu$, $\mathfrak{B}=\langle B, \pi, \cdot, \ast \rangle$ be Fell bundle over $G$ such that each fiber space $B_x$ is separable when $G$ is not discrete. We use symbol $\mathfrak{B} \times \mathfrak{B}$ to denote the Banach bundle over $G \times G$ which is the retraction from $\mathfrak{B}$ by $\phi: G \times G \to G,  (x,y) \mapsto xy^{-1}$ and we denote its bundle space by $D$. If $\tau$ is a $\ast$-representation of $\mathfrak{B}$, we use symbol $\tau_{\rm{int}}$ to denote the integrated form of $\tau$. We use symbol $\mathfrak{}$

Sometimes we explicitly stated that $G$ is discrete or not, but in general we use unified notion to describe the both cases. For instance, for a function $f: G \to X$ into locally convex space $X$, we use symbol $\int_G f(x) d \mu (x)$ to denote integration or summation if $G$ is either non-discrete or is discrete.

\begin{remark}\label{cvnhjfuk}
In our context if $G$ is not discrete then by \cite[Proposition II.13.21]{MR936628} the bundle space $B$, and so $D$ are both second countable.
\end{remark}
\bigskip

We give some conventions which will be assumed in the rest of this thesis without reference.

\bigskip

The following lemma is trivial, we list it here for reference:

\begin{lemma}\label{hfdiosd}
Let  $\mathfrak{D}$ be a Banach bundle with bundle space $B'$ over locally compact Hausdorff space $N$, let $\mathfrak{C}$ be the Banach bundle over locally compact Hausdorff space $M$ which is the retraction of $\mathfrak{D}$ by a continuous map $\phi: M \to N$, and we denote its bundle space by $Z$. If $f: M \to B'$ is a map such that $f(m) \in B'_{\phi(m)}$ for all $m \in M$, then  the map $\widetilde{f}: M \to Z$ defined by
\begin{equation*}
\widetilde{f}(m)=(m,f(m)) \in Z_m \ \ \ \ (m \in M)
\end{equation*}
is a cross-sectional $\mathfrak{C}$, and that $f$ is continuous if and only if $\widetilde{f}$ is continuous. We call $\widetilde{f}$ the cross-section of $\mathfrak{C}$ by canonical identification of $f$.
\end{lemma}

In the following, we will identify $f$  and its canonical identification $\widetilde{f}$, i.e if $f: M \to B'$ is a (continuous) map such that $f(m) \in B'_{\phi(m)}$, we will regard that $f$ is (continuous) cross-section of $\mathfrak{C}$.

Therefore for any continuous map $k: G \times G \to B$  with $k(s,t) \in B_{st^{-1}}$, we could regard $k$ as a continuous cross-section of $\mathfrak{B} \times \mathfrak{B}$. 
\bigskip

Recall that the bundle space $D$ of $\mathfrak{B} \times \mathfrak{B}$ is a (topological) subspace of $G \times G \times B$, i.e
\begin{equation*}
D=\{(s,t; a): s,t \in G,\  a \in B_{st^{-1}}\}.
\end{equation*}
and each fiber space $D_{s,t}=\{(s,t; a): a \in B_{st^{-1}}\}$ of $\mathfrak{B} \times \mathfrak{B}$ is isometric to $B_{st^{-1}}$ $(s,t \in G)$. Therefore, we have the following lemma which is easy consequence of the definition of the topology of $D$.

Let denote the subspace  $\{[(s,t; a), (t,r; b)]: s,t,r \in G\}$ of $D \times D$ by $\mathcal{Z}$,  we define  $\diamond: \mathcal{Z} \to D$ by  (notice that $ab \in B_{sr^{-1}}$ for $a \in B_{st^{-1}}$ and $b \in B_{tr^{-1}}$)
\begin{equation}\label{19051301}
(s,t; a) \diamond (t,r; b)=(s,r; ab) \in D_{s,r} 
\end{equation}
 furthermore define $\ast: D \to D$ by (here notice that $a^{\ast} \in B_{ts^{-1}}$ for $a \in B_{st^{-1}}$)
\begin{equation}\label{19051302}
(s,t; a)^{\ast}=(t,s; a^{\ast}) \in D_{t,s}.
\end{equation}
since the multiplication and involution are continuous in $B$, of course both $\diamond$ and $\ast$ are continuous.

Let $\rho : B \to  \mathcal{O}(X)$ be any fixed non-degenerate $\ast$-representation on Hilbert space $X$ such that $\rho|_{B_e}$ is faithful. We define $\rho_D: D \to \mathcal{O}(X(\rho))$ by 
\begin{equation}\label{hvcdffiiul}
\rho_D((s,t; a)):=\rho(a) \ \ \ \  ((s,t; a) \in D),
\end{equation}
 we have
\begin{equation}\begin{split}\label{19051303}
\rho_D((s,t; a) \diamond  (t,r; b))=\rho_D((s,t; a)) \rho_D((t,r; b)) \ \ \ \ ((s,t; a), (t,r; b) \in D)
\end{split}\end{equation}
and
\begin{equation}\begin{split}\label{19051304}
\rho_D((s,t; a))^{\ast})=\rho_D((s,t; a))^{\ast} \ \ \ \ (s,t \in G; \ (s,t; a) \in D_{s,t}).
\end{split}\end{equation}
Furthermore, since $\rho: B \to \mathcal{O}(X(\rho))$ is strong operator continuous,  for any $k \in \mathfrak{L} (\mathfrak{B} \times \mathfrak{B} )$ the map $(s,t) \mapsto \rho_D(k(s,t))$ is strong-operator continuous map from $G \times G$ into $\mathcal{O}(X)$.

\bigskip
Recall that the bundle space $B$ of $\mathfrak{B}$ could be identified with a subset of $M(C^{\ast}(\mathfrak{B}))$ such that the topology of $B$ is stronger than the relativized topology of $B$ defined by the strict topology $M(C^{\ast}(\mathfrak{B}))$ (see \cite{MR936629} and \cite{MR1891686}). Thus for $a \in B$ and $b \in C^{\ast}(\mathfrak{B})$, we have multiplications $ab$ and $ba$ in $M(C^{\ast}(\mathfrak{B}))$ such that $ab \in C^{\ast}(\mathfrak{B})$ and $ba \in C^{\ast}(\mathfrak{B})$. We define $\star: D \times C^{\ast}(\mathfrak{B}) \to C^{\ast}(\mathfrak{B})$ by
\begin{equation*}
(s,t; a) \star b=ab  \ \ \ \ ((s,t; a) \in D, b \in C^{\ast}(\mathfrak{B})).
\end{equation*}
For the sake of convenience, we use the same use the same symbol $\star$ to denote the map from $C^{\ast}(\mathfrak{B}) \times D$ to $C^{\ast}(\mathfrak{B})$ defined by
\begin{equation*}
b \star (s,t; a)=b a  \ \ \ \ ((s,t; a) \in D, b \in C^{\ast}(\mathfrak{B})).
\end{equation*}
 Thus we have
 \begin{equation*}\begin{split}
 \rho_{\rm{int}}((s,t; a) \star b=\rho_D((s,t; a)) \rho_{\rm{int}}(b) \ \ \ \ ((s,t; a) \in D; b \in C^{\ast}(\mathfrak{B} )),
\end{split}\end{equation*}
similarly we have
\begin{equation*}
 \rho_{\rm{int}}(b \star (s,t; a))= \rho_{\rm{int}}(b)\rho_D((s,t; a))  \ \  ((s,t; a) \in D; b \in C^{\ast}(\mathfrak{B} )).
\end{equation*}
If $k \in \mathfrak{L}(\mathfrak{B} \times \mathfrak{B})$ and $b \in C^{\ast}(\mathfrak{B})$, then the maps $(s,t) \mapsto k(s,t) \star b$ and $(s,t) \mapsto b \star k(s,t) $ are continuous from $G \times G$ into $C^{\ast}(\mathfrak{B})$.

\bigskip
Our final remark is about the integration theory of Banach bundles. Let $\mathfrak{C}$ be an arbitrary Banach bundle over locally compact Hausdorff space $M$ with Borel measure $\nu$. We use symbol
 $\mathfrak{R}(\mathfrak{C})$ to denote the set of the linear span of 
 \begin{equation*}
 \{\chi_W k':\ k' \in  \ \mathfrak{L}(\mathfrak{C})\ {\rm{and}} \ W {\rm{ \ is \ compact \ subset \ of }} \ M \}.
 \end{equation*}
By \cite[Chapter II]{MR936629}, for any $k \in \mathfrak{L}_p(\mathfrak{C})$, there is a sequence $\{k_n\}_{n \in \mathbb{N}} \subset \mathfrak{R}(\mathfrak{C})$ such that
\begin{equation*}
\|k_n(x)-k(x)\| \to 0 \ \ \ \ ({\rm{provided}} \ n \to \infty)
\end{equation*}
and
\begin{equation*}
\|k_n(x)\| \leq \|k(x)\| \ \ \ \ (n \in \mathbb{N})
\end{equation*}
for almost $x \in M$. In particular since $\|k_n(x)-k(x)\|^p \leq (\|k_n(x)\|+\|k(x)\|)^{p}$ for almost $x \in M$ and that $x \mapsto  (\|k_n(x)\|+\|k(x)\|)^{p}$ is integrable, apply Lebesgue's Dominated Theorem we have
\begin{equation*}
\int_G \|k_n(x)-k(x)\|^p d \nu x \to 0.
\end{equation*}

\section{Schur multipliers of Fell bundles }

In this section we define the notion of Schur multiplier of Fell bundles.

\begin{proposition}\label{fhjsdvnc}
Let $k \in \mathfrak{L}_2 (\mathfrak{B} \times \mathfrak{B})$. For arbitrary $\xi \in \mathfrak{L}_2(G, X)$ the map
 \begin{equation*}
P^{x}_{k, \xi}: G \to X, \ y \mapsto \rho_D(k(x,y))(\xi(y))
 \end{equation*}
 is integrable for almost $x \in G$, and the following map
 \begin{equation*}
Q_{k, \xi}: G \to X, \  x \mapsto \int_G \rho_D(k(x,y))(\xi(y)) dy
 \end{equation*}
 is in $\mathfrak{L}_2(G, X)$.
Therefore, we can associate an operator $T_{k}^{\rho}$ on Hilbert space $\mathfrak{L}_2(G, X)$, defined by
\begin{equation}\label{definition of the operator}
T_k^{\rho}(\xi)(x)= \int_G \rho_D(k(x,y))(\xi(y)) dy \ \ \ \ (\xi \in \mathfrak{L}_2(G, X)),
\end{equation}
and we have $\|T_k^{\rho}\| \leq \|k\|_2$. 

\end{proposition}
\begin{proof}

  Let $k \in \mathfrak{R}(\mathfrak{B} \times \mathfrak{B})$ and suppose $\xi \in \mathfrak{L}_2(G,X)$ has the form of $\xi=\sum_{i=1}^n \chi_{V_i}\xi'_i$ for some compact subsets $V_i \subset G$ and $\xi_i' \in X$. Then it is easy to verify that $P^x_{k,\xi}$  is integrable for almost $x \in G$. Furthermore, by Fubini's Theorem, $Q_{k, \xi}$ is measurable and is compactly supported,  and by
\begin{equation}\begin{split}\label{hfdiowerfuc}
\int_G \left\|Q_{k,\xi}(x)\right\|^2 dx \leq \|k\|_2^2 \|\xi\|_2^2,
\end{split}\end{equation}
we proved that $Q_{k,\xi}$ is in $\mathfrak{L}_2(G, X)$.

 Now let $k$ be an arbitrary element of $\mathfrak{L}_2(\mathfrak{B} \times \mathfrak{B})$ and $\xi$ an arbitrary element of $\mathfrak{L}_2(G,X)$. We prove that $P^x_{k,\xi}$ is integrable for almost $x \in G$. Let $\{\xi_n\}_{n \in \mathbb{N}}$ be sequence of simple functions of $\mathfrak{L}_2(G,X)$ such that $\|\xi_n(x)-\xi(x)\| \to 0$ for almost $x \in G$, and $\|\xi_n-\xi\|_2 \to 0$; on the other hand let $\{k_n\}_{n \in \mathbb{N}}$ be a sequence of $\mathfrak{R}(\mathfrak{B} \times \mathfrak{B})$ such that $\|k_n(x,y)\| \leq \|k(x,y)\|$ and $\|k_n(x,y)-k(x,y)\| \to 0$ in $D_{x,y}$ for almost $(x,y) \in G \times G$ (so automatically $\|k_n-k\|_2 \to 0$). Therefore, we have a null subset $N \subset G \times G$ such that
\begin{equation*}
 \|k_n(x,y)-k(x,y)\| \to 0 
\end{equation*}
and
\begin{equation*}
\|k_n(x,y)\| \leq \|k(x,y)\|
\end{equation*}
 if $(x,y) \in (G \times G)\setminus N$. Now let $N_x=\{y \in G: (x,y) \in N\}$, $M=\{x \in G: N_x {\rm{  \ is \ not \ null \ subset \ of \ }}  G\}$, then $M$ is null subset of $G$, and if $x \in G \setminus M $ we have
 \begin{equation*}
 \|\rho_D(k_n(x,y))(\xi_n(y))-\rho_D(k(x,y))(\xi(y))\| \to 0.
 \end{equation*}
for almost all $y \in G$.  There is a null set $L \subset G$ such that for all $n \in \mathbb{N}$ and $x \in G \setminus L$, the map $y \mapsto \rho_D(k_n(x,y))(\xi_n(y))$ is measurable and vanishes outside compact subsets, thus we conclude that if $x \in G \setminus (M \cup L)$,  the  map $y \mapsto \rho_D(k(x,y))(\xi(y))$ is measurable and vanishing outside countable union of compact subsets. Furthermore, there is null set $L' \subset G$ such that
\begin{equation*}
\int_G \|k(x,y)\|^2 dy < \infty
\end{equation*}
for $x \in G \setminus L'$, thus if $x \in G \setminus (M \cup L \cup L')$ we have
\begin{equation}\begin{split}\label{fhnjdslcfsdkcijlf}
 \int_{G}\|\rho_D(k(x,y))(\xi(y))\| dy\leq \left(\int_G \|\rho_D(k(x,y))\|^2 dy \right)^{1/2} \cdot \|\xi\|_2 < \infty.
\end{split}\end{equation}
 We proved that $P^x_{k, \xi}$ is integrable for almost $x \in G$ (i.e for $x \in G \setminus (M \cup L \cup L')$). 

Now we prove that $Q_{k,\xi}$ is in $\mathfrak{L}_2(G,X)$. Our first step is to verify that $Q_{k,\xi}$ is measurable. For $x \in G \setminus (M \cup L \cup L')$ we have
\begin{equation*}\begin{split}
\ \ \ \ \ \  &\left\|\int_G  (\rho_D(k(x,y))(\xi(y))-\rho_D(k_n(x,y))(\xi_n(y)) dy \right\|
\\& \leq \left(\int_G\|k_n(x,y)-k(x,y)\|^2 dy\right)^{1/2} \cdot \|\xi\|_2+ \left(\int_G\|k_n(x,y)\|^2 dy\right)^{1/2}  \cdot \|\xi-\xi_n\|_2.
\end{split}\end{equation*}
Let us estimate the value of $\int_G \|k_n(x,y)-k(x,y)\|^2 dy$ if $n \to \infty$.  Recall our definitions on null subsets $L, M, L' \subset G $: if $x \in G \setminus (M \cup L)$, then for almost $y \in G$ we have $\|k_n(x,y)-k(x,y)\| \to 0$ provided $n \to \infty$ and $\|k_n(x,y)\| \leq \|k(x,y)\|$ for all $n \in \mathbb{N}$; on the other hand, if $x \in G \setminus L'$ we have $\int_G \|k(x,y)\|^2 dy < \infty$; hence if $x \in G \setminus (M \cup L \cup L')$ we have $\int_G\|k_n(x,y)-k(x,y)\|^2 dy \to 0$. In particular, we can conclude that if $x \in G \setminus (M \cup L \cup L')$ the sequence $\{\int_G \|k_n(x,y)\|^2 dy\}_{n \in \mathbb{N}}$ is bounded. Therefore we have
\begin{equation*}
\left\|\int_G  (\rho_D(k(x,y))(\xi(y))-\rho_D(k_n(x,y))(\xi_n(y)) dy \right\| \to 0
\end{equation*}                 
for $x \in G \setminus (M \cup L \cup L')$. This proved that $Q_{k, \xi}$ is measurable, and it is clear that $Q_{k, \xi}$ is vanishing outside countable union of compact subsets of $G$. Furthermore, by the same argument which derived (\ref{hfdiowerfuc}) we have
\begin{equation*}
\int_G \left\|Q_{k,\xi}(x)\right\|^2 dx \leq  \|k\|_2^2 \|\xi\|_2^2,
\end{equation*}
thus $Q_{k,\xi}$ is in $\mathfrak{L}_2(G,X)$.

The other parts of this proposition is easy consequence of our previous discussion. 
\end{proof}

We use the symbol $\mathfrak{E}(\rho, \mathfrak{B} )$ to denote the norm-closure of the set $\{T_k^{\rho}: k \in \mathfrak{L}_2(\mathfrak{B})\}$ in $\mathcal{O}(\mathfrak{L}_2(G,X))$. 

\begin{lemma}\label{heisalgebra}
If $k_1, k_2 \in \mathfrak{L}(\mathfrak{B} \times \mathfrak{B})$,  the map 
 \begin{equation*}
 G \to D_{x,y}, z \mapsto k_1(x,z) \diamond k_2(z,y)
 \end{equation*}
 is continuous and compactly supported for all $(x,y) \in G \times G$, and the map  
 \begin{equation*}
J: G \times G \to D_{x,y},  (x,y) \mapsto \int_G  k_1(x,z) \diamond k_2(z,y)dz
\end{equation*}  
 is in $\mathfrak{L}(\mathfrak{B} \times \mathfrak{B})$.
\end{lemma}
\begin{proof}
Since $k_1$ and $k_2$ are continuous cross-section, then the map $p: (x,z,y) \mapsto (x,y; c_{k_1}(x,z) c_{k_2}(z,y))$ is continuous from $G \times G \times G$ into $D$, where $c_{k_i}: G \times G \to B$ satisfies $k_i(x,y)=(x,y, c_{k_i}(x,y))$, and is compactly supported, and we have
\begin{equation*}
k_1(x,z) \diamond k_2(z,y)=p(x,y,z),
\end{equation*}
hence the continuity of $z \mapsto k_1(x,z) \diamond k_2(z,y)$ is from the continuity of $p$. Furthermore, we have
\begin{equation*}\begin{split}
\int_G p(x,z,y) dz  &= \int_G (x,y; c_{k_1}(x,z) c_{k_2}(z,y)) dz
\\&= \int_G  k_1(x,z) \diamond k_2(z,y) dz,
\end{split}\end{equation*}
therefore it is sufficient to prove that $(x,y) \mapsto \int_G p(x,z,y) dz$ is continuous from $G \times G$ into $D$. By \cite[II.15.19]{MR936628}, the map
\begin{equation}
(x,y) \mapsto \int_G p(x,z,y) dz
\end{equation}
is continuous cross-section of $\mathfrak{B} \times \mathfrak{B}$. Our proof is complete.
\end{proof}

\begin{definition}
For $k_1, k_2 \in \mathfrak{L}(\mathfrak{B} \times \mathfrak{B})$, we use symbol $k_1 \star k_2 $ to denote the cross-section of $\mathfrak{B} \times \mathfrak{B}$
\begin{equation*}
(x,y) \mapsto \int_G  k_1(x,z) \diamond k_2(z,y)dz \ (\in B_{xy^{-1}}),
\end{equation*}
$(recall \ that \ k_1 \star k_2 \in \mathfrak{L}(\mathfrak{B} \times \mathfrak{B}))$ and use the symbol $k_1^{\ast}$ to denote the cross-section defined by
\begin{equation*}
(x,y) \mapsto k_1(y,x)^{\ast}\ \ (\in B_{xy^{-1}})
\end{equation*}
\end{definition}

The following lemma is verified by routine computation:
\begin{lemma}
For any $k_1, k_2 \in \mathfrak{L}(\mathfrak{B} \times \mathfrak{B})$, we have
\begin{equation*}\begin{split}
&T^{\rho}_{k_1 \star k_2}=T^{\rho}_{k_1} T^{\rho}_{k_2} ,
\\&( T_{k_1})^{\ast}=T^{\rho}_{k_1^{\ast}}.
\end{split}\end{equation*}
\end{lemma}
\bigskip
Therefore, $\mathfrak{E}(\rho, \mathfrak{B} )$ is a $C^{\ast}$-algebra.

\bigskip

Recall that $\rho$ is non-degenerate $\ast$-representation, then the concrete $C^{\ast}$-algebra $\mathcal{O}_c(\mathfrak{L}_2(G)) \otimes \rho_{\rm{int}}(C^{\ast}(\mathfrak{B} ))$ acts on $\mathfrak{L}_2(G,X)$ non-degenerately, and we have the inclusion $M(\mathcal{O}_c(\mathfrak{L}_2(G)) \otimes \rho_{\rm{int}}(C^{\ast}(\mathfrak{B} ))) \subset \mathcal{O}(\mathfrak{L}_2(G,X))$.  In the following, we prove that $\mathfrak{E}(\rho, \mathfrak{B} ) \subset M(\mathcal{O}_c(\mathfrak{L}_2(G)) \otimes \rho_{\rm{int}}(C^{\ast}(\mathfrak{B} )))$.

\bigskip

We denote the set of all simple functions $l \in \mathfrak{L}_2(G\times G, C^{\ast}(\mathfrak{B} ))$ of the form
\begin{equation}\label{standard simple}
l=\sum_{i=1}^n a_i \chi_{E_i \times F_i}, \ ( a_i \in C^{\ast}(\mathfrak{B} ), \ E_i ,F_i \subset G \mbox{ are  compact})
\end{equation}
by $\mathfrak{K}({\mathfrak{B}})$.
By \cite[II.9.2]{MR936628} and \cite[II.9.4]{MR936628}, the linear span of 
\begin{equation}
\{\chi_{E \times F}: E \ {\rm{and}} \ F \ {\rm{are}} \ {\rm{compact}}\}
\end{equation}
is dense in $\mathfrak{L}_2(G)$, then it is clear that $\mathfrak{K}(\mathfrak{B})$ is dense in $\mathfrak{L}_2(G \times G, C^{\ast}(\mathfrak{B}))$.

\bigskip
Recall from \cite{MTT16} that for any $k' \in \mathfrak{L}_2(G \times G, C^{\ast}(\mathfrak{B}))$, the operator $T_{k'}^{\rho_{\rm{int}}}$ on $\mathfrak{L}_2(G, X)$ defined by
\begin{equation*}
T_{k'}^{\rho_{\rm{int}}}(\xi)(x)=\int \rho_{\rm{int}}(k'(x,y)) (\xi(y)) dy \ \ \ \ (\xi \in \mathfrak{L}_2(G, X))
\end{equation*}
is in $\mathcal{O}_c(\mathfrak{L}_2(G)) \otimes \rho_{\rm{int}}(C^{\ast}(\mathfrak{B} ))$. 

\begin{lemma}\label{maintool}
Let $k \in \mathfrak{L}(\mathfrak{B} \times \mathfrak{B})$ and $l \in \mathfrak{K}({\mathfrak{B} })$.  The following maps from $G$ into $C^{\ast}(\mathfrak{B})$
\begin{equation*}\begin{split}
&  z \mapsto k(x,z) \star l(z,y),
\\&   z \mapsto    l(z,y) \star k(x,z) 
\end{split}\end{equation*}
are in $\mathfrak{L}_1(G, C^{\ast}(\mathfrak{B}))$ for almost $(x,y) \in G \times G$, and the maps $k \star l$ and $l \star k$ defined by 
\begin{equation*}\begin{split}
 k \star l(x,y)=\int_G k(x,z) \star l(z, y) dz \ \ \ \ (x,y \in G), &
\\ l \star k(x,y)=\int_G l(x,z) \star k(z, y) dz \ \ \ \ (x,y \in G)
\end{split}\end{equation*}
are in $\mathfrak{L}_2(G \times G, C^{\ast}(\mathfrak{B}))$. Furthermore we have
\begin{equation*}
T_k^{\rho} T_l^{\rho_{\rm{int}}}=T_{k \star l}^{\rho_{\rm{int}}};   T_l^{\rho_{\rm{int}}} T_k^{\rho} =T_{l \star k}^{\rho_{\rm{int}}},
\end{equation*}

In particular,  since $\{T_l^{\rho_{\rm{int}}}: l \in \mathfrak{K}(\mathfrak{B})\}$ is dense in  $\mathcal{O}_c(\mathfrak{L}_2(G)) \otimes \rho_{\rm{int}}(C^{\ast}(\mathfrak{B} ))$, we have $\mathfrak{E}(\rho, \mathfrak{B} ) \subset M (\mathcal{O}_c(\mathfrak{L}_2(G)) \otimes \rho_{\rm{int}}(C^{\ast}(\mathfrak{B} )))$.
\end{lemma}
\begin{proof}
Let $l$ be a function with the form of (\ref{standard simple}). Suppose ${\rm{supp}}(k) \subset E \times F$, where $E$ and $F$ are compact subsets of $G$. Our first task is to prove that the map $z \mapsto  k(x,z) \star l(z, y) $ is in $\mathfrak{L}_1(G, C^{\ast}(\mathfrak{B}))$ for almost $(x,y) \in G \times G$, then prove that $(x,y) \mapsto \int_G  k(x,z) \star l(z, y) dz$ is in $\mathfrak{L}_2(G, C^{\ast}(\mathfrak{B} ))$. Finally we verify $T_k^{\rho} T_l^{\rho_{\rm{int}}}=T_{k \star l}^{\rho_{\rm{int}}}$.

Let define $q: G \times G \times G \to C^{\ast}(\mathfrak{B} )$ by
\begin{equation*}
q(x,z,y)=k(x,z) \star l(z,y)=\sum_{i=1}^n (k(x,z) \star a_i) \chi_{E_i \times F_i}(z,y).
\end{equation*}
then $q$ is bounded, measurable and compactly supported by $E \times (\bigcup_{i=1}^nE_i) \times (\bigcup_{i=1}^nF_i) $. In particular, $q \in \mathfrak{L}_1(G \times G  \times  G, C^{\ast}(\mathfrak{B} ))$, thus by Fubini's Theorem the map $z \mapsto  k(x,z) \star l(z, y) $ is in $\mathfrak{L}_1(G, C^{\ast}(\mathfrak{B}))$ for almost  $(x,y) \in G \times G$. 

For any fixed $x,y$, we define $p_{x,y}: G \to  C^{\ast}(\mathfrak{B} )$ by $p_{x,y}(z)=q(x,z,y)(=k(x,z) \star l(z,y))$, then by the fact that $q \in \mathfrak{L}_1(G \times G  \times  G, C^{\ast}(\mathfrak{B} ))$ and Fubini's Theorem, the map $k \star l:(x,y) \mapsto \int_G p_{x,y}(z) dz$ is measurable. On the other hand, since $q$ is compactly supported supported by $E \times (\bigcup_{i=1}^nE_i) \times (\bigcup_{i=1}^nF_i) $ and bounded, we conclude that $k \star l: (x,y) \mapsto \int_G p_{x,y}(z) dz=\int_G q(x,z,y) dz$ is bounded, measurable and compactly supported by $E \times \bigcup_{i=1}^nF_i$. Therefore, this map is square-integrable, we proved that $k \star l \in \mathfrak{L}_2(G \times G, C^{\ast}(\mathfrak{B} ))$.

By our previous discussion, for almost $(x,y) \in G \times G$ we have $C^{\ast}(\mathfrak{B})$-valued integration $\int_G k(x,z) \star l(z,y) dz$,  thus by the norm contunuity of $\rho_{\rm{int}}: C^{\ast}(\mathfrak{B}) \to \mathcal{O}(X)$ we have
\begin{equation*}
\rho_{\rm{int}} \left(\int_G k(x,z) \star l(z,y) dz \right)=\int_G \rho_{\rm{int}}(k(x,z) \star l(z,y)) dz,
\end{equation*}
 for almost $(x,y) \in G \times G$, hence for any fixed $\xi \in \mathfrak{L}_2(G,X)$, for almost $x \in G$ we have
\begin{equation*}\begin{split}
T_{k \star l}^{\rho_{\rm{int}}} \xi(x)&=\int_G \rho_{\rm{int}} \text{\LARGE{(}}\int_G k(x,z) \star l(z,y) dz \text{\LARGE{)}}\xi(y) dy
\\&= T_k^{\rho} T_l^{\rho_{\rm{int}}}\xi (x).
\end{split}\end{equation*}
Thus we have proved that $T_{k \star l}^{\rho_{\rm{int}}}= T_k^{\rho} T_l^{\rho_{\rm{int}}}$. $ T_l^{\rho_{\rm{int}}} T_k^{\rho} =T_{l \star k}^{\rho_{\rm{int}}}$ may be proved by the same argument.
\end{proof}

\bigskip

Recall that for any $C^{\ast}$-algebra $A$, if $S: A \to \mathcal{O}(X)$ is a non-degenerate $^{\ast}$-representation of $A$ on Hilbert space $X$, then $S$ can be uniquely extended to $M(A)$. Therefore we can regard $T$ as a non-degenerate $^{\ast}$-representation of $M(A)$ on $X$.

\begin{lemma}\label{19011401}
Let $r: B \to \mathcal{O}(Y)$ be a non-degenerate $\ast$-representation of $\mathfrak{B}$ on Hilbert space $Y$ which is weakly contained $\rho$. Then the following map
\begin{equation*}
\Xi_{\rho, r}: T_{k'}^{\rho_{{\rm{int}}}} \mapsto T_{k'}^{r_{{\rm{int}}}} \ \ \ \ (k' \in \mathfrak{L}_2(G \times G, C^{\ast}(\mathfrak{B})))
\end{equation*}
can be extended to a $\ast$-homomorphism from $\mathcal{O}_c(\mathfrak{L}_2(G)) \otimes \rho_{{\rm{int}}}(\mathfrak{B})$ onto $\mathcal{O}_c(\mathfrak{L}_2(G)) \otimes r_{{\rm{int}}}(\mathfrak{B})$. Therefore we can regard  $\Xi_{\rho, r}$ as a $^{\ast}$-homomorphism from $M(\mathcal{O}_c(\mathfrak{L}_2(G)) \otimes \rho_{{\rm{int}}}(\mathfrak{B}))$ onto  $\mathcal{O}_c(\mathfrak{L}_2(G)) \otimes r_{{\rm{int}}}(\mathfrak{B})$, and we have 
\begin{equation}\label{nvjhdffhugv}
\Xi_{\rho, r}(T_k^{\rho})=T_k^{r}  \ \ \ \ (k \in \mathfrak{L}_2(\mathfrak{B} \times \mathfrak{B}))
\end{equation}
and
\begin{equation}\label{fcnhnjdsidf}
\Xi_{\rho, r}(T_k^{\rho}) \Xi_{\rho, r}(T_l^{\rho_{{\rm{int}}}})=T_k^{r} T_{l}^{r_{{\rm{int}}}} \ \ \ \ (k \in \mathfrak{L}(\mathfrak{B} \times \mathfrak{B}); \ l \in \mathfrak{K}({\mathfrak{B} }))
\end{equation}
\end{lemma}
\begin{proof}
Since $r$ is weakly contained in $\rho$, the map $S: \rho_{\rm{int}}(C^{\ast}(\mathfrak{B})) \to r_{\rm{int}}(C^{\ast}(\mathfrak{B}))$ defined by $S( \rho_{\rm{int}}(a))=r_{\rm{int}}(a)$ $(a \in C^{\ast}(\mathfrak{B}))$ is $^{\ast}$-homomorphism. Now $S \otimes I_{\mathfrak{L}_2(G)}$ is $^{\ast}$-homomorpphism which is extension of $\Xi_{\rho,r}$.

Therefore for each $k \in \mathfrak{L}(\mathfrak{B} \times \mathfrak{B})$ and $l \in  \mathfrak{K}({\mathfrak{B} })$, by Lemma \ref{maintool} we have
\begin{equation*}\begin{split}
((S \otimes I_{\mathfrak{L}_2(G)})(T_k^{\rho})) \ (S \otimes I_{\mathfrak{L}_2(G)}(T_l^{\rho_{\rm{int}}}))= T_k^{r} \circ ((S \otimes I_{\mathfrak{L}_2(G)})(T_l^{\rho_{\rm{int}}})) ,
\end{split}\end{equation*}
since $\{T_l^{\rho_{\rm{int}}}: l \in  \mathfrak{K}({\mathfrak{B} })\}$ is dense in $\mathcal{O}_c(Y) \otimes r_{\rm{int}}(C^{\ast}(\mathfrak{B}))$ and $S \otimes I_{\mathfrak{L}_2(G)}$ is non-degenerate, we conclude that
\begin{equation}\label{cnnvhgdfu}
\Xi_{\rho, r}(T_k^{\rho})=(S \otimes I_{\mathfrak{L}_2(G)})(T_k^{\rho})= T_k^{r} \ \ \ \ (k \in \mathfrak{L}(\mathfrak{B} \times \mathfrak{B})).
\end{equation} 
Now (\ref{nvjhdffhugv}) is proved for $k \in \mathfrak{L}(\mathfrak{B} \times \mathfrak{B})$. By that $\mathfrak{L}(\mathfrak{B})$ is dense in $\mathfrak{L}_2(\mathfrak{B})$, (\ref{nvjhdffhugv}) holds for $k \in \mathfrak{L}_2(\mathfrak{B} \times \mathfrak{B})$.

(\ref{fcnhnjdsidf}) is direct consequence of Lemma \ref{maintool}.

The proof is complete.
\end{proof}

\begin{proposition}
Let $k \in \mathfrak{L}_2(\mathfrak{B} \times \mathfrak{B})$. Then $T^{\rho}_k$=0 if and only if $k(x,y)=0$ for almost all $(x,y) \in G \times G$. 
\end{proposition}
\begin{proof}
Since $\rho_{\rm{int}}(C^{\ast}(\mathfrak{B}))$ is separable $C^{\ast}$-algebra, let  $R: \rho_{\rm{int}}(C^{\ast}(\mathfrak{B})) \to \mathcal{O}(Y) $ be a non-degenerate faithful $\ast$-representation of $\rho_{\rm{int}}(C^{\ast}(\mathfrak{B}))$ on separable Hilbert space Y, and let $r:B \to \mathcal{O}(Y)$ be the $^{\ast}$-representation of $\mathfrak{B}$ on $Y$ such that $R \circ \rho_{\rm{int}}$ is the integrated form of $r$. Then $r$ is weakly equivalent to $\rho$. By \cite{MTT16}$ T^r_k=0$ if and only if $k(x,y)=0$ for almost all $(x,y) \in G \times G$. 

Now we prove that $T^{\rho}_k=0$ implies that $k(x,y)=0$ almost everywhere. Since $\rho$ and $r$ are weakly equivalent, then the map $\Xi_{\rho, r}$ is faithful $^{\ast}$-homomorhpism,  hence $T_k^{\rho}=0$ implies that $T^r_k=0$. On the other hand, by our discussion in the last paragraph, we conclude that $T_k^{\rho}=0$ implies that $\|k(x,y)\|=0$ almost everywhere. Our proof is complete.
\end{proof}

\bigskip

\begin{definition}\label{chsduiiusrfcnh}
Let $\mathfrak{D}$ be a Banach bundle over locally compact space M with bundle space $D'$. We call a continuous map $\Phi: D' \to D'$ multiplier of $\mathfrak{D}$ if $\Phi$ satisfies: 

$(i)$ $\Phi(D'_{x}) \subset D'_x$ for all $x \in M$; 

$(ii)$ $\Phi|D'_x$ is bounded linear map such that ${\rm{sup}}_{x \in M} \|\Phi|D'_x\| < \infty$. 
\end{definition}

Let $\Phi: D \to D$ be multiplier of $\mathfrak{B} \times \mathfrak{B}$. We denote ${\rm{sup}}_{(x,y) \in G \times G}\|\Phi|D_{x,y}\|$ by $M$. For any $k \in \mathfrak{L}_2 (\mathfrak{B} \times \mathfrak{B})$ we define $\Phi \cdot k (x,y)=\Phi(k(x,y))$ $(x,y \in G)$, then $\Phi \cdot k  \in \mathfrak{L}_2 (\mathfrak{B} \times \mathfrak{B})$.

\begin{definition}
Let $\Phi: D \to D$ be multiplier of $\mathfrak{B} \times \mathfrak{B}$. If the map ${S_{\rho}},_{\Phi}: \{T_k^{\rho}: k \in \mathfrak{L}_2 (\mathfrak{B} )\} \to \mathcal{O}(\mathfrak{L}_2(G,X))$ which is defined by
\begin{equation*}
{S_{\rho}},_{\Phi}(T_k^{\rho})=T_{\Phi \cdot k}^{\rho}
\end{equation*}
is completely bounded linear map, then we say $\Phi$ is Schur $(\mathfrak{B} ,\rho)$-multiplier. Furthermore, if ${S_{\rho}},_{\Phi}$ is completely positive, we say $\Phi$ is completely positive Schur $(\rho, \mathfrak{B})$-multiplier.

If $\Phi$ is Schur $(\mathfrak{B} ,\rho)$-multiplier, we define
\begin{equation*}
\|\Phi\|_{\mathfrak{S}, \rho}:=\|{S_{\rho}},_{\Phi}\|_{\rm{cb}}.
\end{equation*}
\end{definition}

Therefore, if $\Phi$ is Schur $(\mathfrak{B} , \rho)$-multiplier, $S_{\rho,\Phi}$ can be extended to a completely bounded map from $\mathfrak{E}(\rho, \mathfrak{B} )$ into $\mathcal{O}(\mathfrak{L}_2(G, X))$. Thus we will always consider that $S_{\rho,\Phi}$ is a completely bounded map defined on $\mathfrak{E}(\rho, \mathfrak{B} )$.

The following proposition is useful in our further study:

\begin{proposition}\label{19012303}
If $r: B \to \mathcal{O}(Y)$ is a non-degenerate $\ast$-representation of $\mathfrak{B} $ on a Hilbert space $Y$ which is weakly equivalent to $\rho$, then $\Phi$ is Schur $(\mathfrak{B} , r)$-multiplier if and only if it is $(\mathfrak{B} , \rho)$-multiplier.
\end{proposition}
\begin{proof}
Since $r$ is weakly to $\rho$, by Lemma \ref{19011401}, the restriction of $\Xi_{\rho, r}$ on $\mathfrak{E}(\rho, \mathfrak{B})$, i.e
\begin{equation*}
\Xi_{\rho, r}|_{\mathfrak{E}(\rho, \mathfrak{B})}: \mathfrak{E}(\rho, \mathfrak{B}) \to \mathfrak{E}(r, \mathfrak{B}), T^{\rho}_k \mapsto T^{r}_k \ \ \ \ (k \in \mathfrak{L}(\mathfrak{B} \times \mathfrak{B})),
\end{equation*} 
is $\ast$-isoomomorphism from $\mathfrak{E}(\rho, \mathfrak{B})$ onto $\mathfrak{E}(r, \mathfrak{B})$ with inverse 
$\Xi_{r, \rho}|_{\mathfrak{E}(r, \mathfrak{B})}$.

Suppose that $\Phi$ is Schur $(\mathfrak{B}, \rho)$-multiplier, then ${S_{\rho}},_{\Phi}: \mathfrak{E}(\rho, \mathfrak{B}) \to \mathfrak{E}(\rho, \mathfrak{B})$ is completely bounded. Thus $\Xi_{\rho, r} \circ {S_{\rho}},_{\Phi} \circ \Xi_{\rho, r}^{-1}: \mathfrak{E}(r, \mathfrak{B}) \to \mathfrak{E}(r, \mathfrak{B})$ is completely bounded, and
$\Xi_{\rho, r} \circ {S_{\rho}},_{\Phi} \circ \Xi_{\rho, r}^{-1}=S_{r, \Phi}$, thus $\Phi$ is Schur $(\mathfrak{B} , r)$-multiplier. By the same argument, we may prove that if $\Phi$ is  Schur $(\mathfrak{B} , r)$-multiplier, then $\Phi$ is  Schur $(\mathfrak{B} , \rho)$-multiplier.

\end{proof}

\section{Characterization of Schur $(\mathfrak{B}, \rho)$-multipliers}

Our main result of this section is the characterization of Schur $(\mathfrak{B}, \rho)$- multipliers. For this ourpose we need to define a specific class of completely bounded maps. In the following  let $A$ be a fixed $C^{\ast}$-algebra, $M(A)$ be its multiplier algebra, and $X$ be a fixed Hilbert space.

\begin{definition}
 By a completely bounded map from $M(A)$ into $\mathcal{O}(X)$ with Property $(S_A)$ or briefly a $(S_A)$-map on M(A), we shall mean a completely bounded  linear map $f: M(A) \to \mathcal{O}(X)$ which is continuous with respect to the strict topology of $M(A)$ and strong$^{\ast}$ topology of $\mathcal{O}(X)$ on the norm bounded subsets of $M(A)$.

Let $C \subset M(A)$ be $C^{\ast}$-subalgebra. By a completely bounded $(S_{C,A})$-extendable map from $B$ into $\mathcal{O}(X)$ or briefly a $(S_{C,A})$-extendable
map on $C$, we shall mean a completely bounded map $g: C \to \mathcal{O}(X)$ which has an extension of $(S_A)$-map. We say that such extension is an $(S_{C,A})$-extension of $g.$
\end{definition}

It is easy to see that if $r: M(A) \to \mathcal{O}(Y)$ is a $\ast$-representation of $M(A)$ on Hilbert space $Y$,  then there is a $r$-stable space $Z \subset Y$ such that $r(A)(Z^{\bot})=\{0\}$ and $r(A)$ acts on $Z$ non-degenerately, by this fact we have

\begin{lemma}\label{19011601}
Let $f: M(A) \to \mathcal{O}(X)$ be a completely bounded map, then $f$ is $(S_A)$-map on M(A) if and only if f has a representation $(W, V, r, Y)$ such that $r|_A: A \to \mathcal{O}(Y)$ is non-degenerate. If $A$ and $X$ are separable, then $Y$ can be chosen to be separable.
\end{lemma}
\begin{proof}
Suppose $f$ is $(S_A)$-map on $M(A)$, then let $(W',V',r', Y')$ be non-degenerate representation of $f$ (notice that if $A$ and $X$ are separable, then $Y'$ can be chosen to be separable \cite[page 45]{Paulsen}), then there is a $r'$-stable space $Z \subset Y'$ such that $r'(A)(Z^{\bot})=\{0\}$ and $r'(A)$ acts on $Z$ non-degenerately. Let $E$ be the projection of $Y'$ onto $Z$, then for any $\xi \in X$ and $a \in M(A)$ we have
\begin{equation*}
f(a)(\xi)=W'^{\ast} r'(a) V'(\xi)=W'^{\ast} r'(a) E V'(\xi) + W'^{\ast} r'(a) (I- E) V'(\xi).  
\end{equation*}
Let define $r: M(A) \to \mathcal{O}(Z)$ by $r(a)=Er'(a)E$ for all $a \in M(A)$ (recall that $E$ is in the commutator algebra of $r'$ because $Z$ is $r$-stable), then $r$ is non-degenerate $\ast$-representation of $M(A)$ such that $r|_A$ is non-degenerate. Let $\{a_i\}_{i \in I} \subset A$ be a norm-bounded net such that $a_i \to a$ strictly, then $r(a_i) \to r(a)$ in strong$^{\ast}$-operator topology of $\mathcal{O}(Z)$  provided $i \to \infty$, we have
\begin{equation*}\begin{split}
f(a)(\xi)&={\rm{lim}}_{i \to \infty}f(a_i)(\xi)
\\&=W'^{\ast}Er(a)EV'(\xi) \ \ \ \ (\xi \in X),
\end{split}\end{equation*}
thus if we define $W=EW'$, $V=EV'$, then we have
\begin{equation*}
f(a)=W^{\ast}r(a)V \ \ \ \ (a \in M(A)).
\end{equation*}
Therefore if we set $Y=Z$, then $(W, V, r, Y)$ is representation of $f$ such that $r(A)$ acts on $Z$ non-degenerately.

Conversely, suppose $f$ has non-degenerate representation $(W,V,r, Y)$ such that $r|_A$ is non-degenerate $\ast$-representation of $A$ on $Y$, then let $\{a_i\}_{i \in I}$ be a norm-bounded net of elements of $A$ such that $a_i \to a$ in $M(A)$ strictly, we have
\begin{equation*}\begin{split}
{\rm{lim}}_{i \to \infty}f(a_i)(\xi)=f(a)(\xi) \ \ \ \ (\xi \in X).
\end{split}\end{equation*}
On the other hand, by the same argument we have
\begin{equation*}
f(a)^{\ast}(\xi)={\rm{lim}}_{i \to \infty}f(a_i)^{\ast}(\xi) \ \ \ \ (\xi \in X).
\end{equation*}
We conclude that $f(a_i) \to f(a)$ in strong$^{\ast}$ topology of $\mathcal{O}(X)$. Our proof is complete.
\end{proof}

\begin{corollary}\label{190130}
Let $f: A \to \mathcal{O}(X)$ be completely bounded map, then $f$ has unique $(S_{A,A})$-extension $\widetilde{f}$ on $M(A)$, which is determined by
\begin{equation}\label{19013001}
\widetilde{f}(a)={\rm{lim}}_{i \to \infty}f(a_i) , 
\end{equation}
where $a \in M(A)$ and $\{a_i\}_{i \in I} \subset A$ is a norm-bounded net converging to $a$ strictly, and the right side {\rm{limit}} is in strong$^{\ast}$-topology of $\mathcal{O}(X)$.
\end{corollary}
\begin{proof}
Let $(W, V, r, Y)$ be non-degenerate representation of $f$, then we extend $r$ to non-degenerate $^{\ast}$-representation of $M(A)$ which we still denote by $r$, and define $\widetilde{f}: M(A) \to \mathcal{O}(X)$ by
\begin{equation*}
\widetilde{f}(a)=W^{\ast}r(a)V \ \ \ \ (a \in M(A)).
\end{equation*}
Then $\widetilde{f}$ is extension of $f$ and by Proposition \ref{19011601} $\widetilde{f}$ is $(S_A)$-map on $M(A)$. Furthermore, since $A$ is dense in $M(A)$ in the strict topology, then $(\ref{19013001})$ is obvious.
\end{proof}

The following corollary contains an alternative and elementary proof of non-unital version of Stinepring's Theorem:
\begin{corollary}\label{19041501}
Let $f: A \to \mathcal{O}(X)$ be a completely positive map, then

(i) f has a unique $(S_{A,A})$-extension $\widetilde{f}$ which is completely positive; 

(ii) f  has representation $(V, R)$;

(iii) let $\{a_i\}_{i \in I}$ be any approximate unit of $A$ which satisfies $\|a_i\| \leq 1$, then we have
\begin{equation*}
\|f\|_{\rm{cb}}={\rm{lim}}_{i \to \infty}\|f(a_i)\|=\|\widetilde{f}\|_{{\rm{cb}}}.
\end{equation*}
\end{corollary}
\begin{proof}
$(i)$: By Corollary \ref{190130}, $f$ has unique $(S_{A,A})$-extension on $M(A)$. Furthermore, by $(\ref{19013001})$ $\widetilde{f}$ is completely positive, $(i)$ is proved. 

$(ii)$: Since $M(A)$ is unital, we could apply  the unital version Stinepring's Theorem to $\widetilde{f}$ to get representation $(V,R)$, then we have
\begin{equation*}
f(a)=V^{\ast}R(a) V \ \ \ \ (a \in A).
\end{equation*}
Therefore $(V,R|_A)$ is a representation of $f$, $(ii)$ is proved.

$(iii)$: Let $\{a_i\}_{i \in I}$ be an approximate unit of $A$, then since $\widetilde{f}$ is $(S_A)$-map on $M(A)$ and completely positive, we have
\begin{equation}\label{19013002}
\widetilde{f}(1_{M(A)})(\xi)={\rm{lim}}_{i \to \infty}\widetilde{f}(a_i)(\xi) \ \ \ \ (\xi \in X),
\end{equation}
it implies that
\begin{equation}\label{jfnvsfdnvker}
\|\widetilde{f}(1_{M(A)})\| \leq  \underline{{\rm{lim}}}_{\, i \to \infty} \|\widetilde{f}(a_i)\|
\end{equation}
On the other hand we notice that $\|\widetilde{f}(a_i)\| \leq \|\widetilde{f}\|_{\rm{cb}}=\|\widetilde{f}(1_{M(A)})\|$, so (\ref{jfnvsfdnvker}) implies that $\|\widetilde{f}\|_{\rm{cb}}=\|f\|_{\rm{cb}}$.
our proof is complete.
\end{proof}

\begin{corollary}\label{19012801}
If $C \subset A$ and  $f: C \to \mathcal{O}(X)$ is completely bounded, then f is $(S_{C,A})$-extendable map. Furthermore, if $f$ is completely positive, then the extension of $f$ may be chosen to be completely positive.
\end{corollary}
\begin{proof}
We first extend $f$ to a completely bounded map from $A$ into $\mathcal{O}(Y)$, then by Corollary \ref{190130} we conclude that $f$ has $(S_{C,A})$-extension.

Now suppose that $f$ is completely positive. By Corollary \ref{19041501} we have non-degenerate representation $(V,r)$ of $f$. Let extend $r$ to non-degenerate $\ast$-representation $\widetilde{r}$ of
 $C^{\ast}$-algebra $C'=\{b + t 1_{M(A)}: b \in B, t  \in \mathbb{C}\}^{-\|\cdot\|_{M(A)}}$. Let define $g': C' \to \mathcal{O}(X)$ by
\begin{equation*}
g'(a)=V^{\ast}\widetilde{r}(a)V \ \ \ \ (a \in C'),
\end{equation*}
then $g'$ is completely positive. Now since $B'$ is a $C^{\ast}$-subalgebra of $A'=\{a +t 1_{M(A)}: a \in A; t \in \mathbb{C}\}^{-\|\cdot\|_{M(A)}}$ such that $C'$ contains the unit of $A'$ (i.e $1_{M(A)}$), we can extend $g'$ to a completely positive map $g$ from $A'$ into $\mathcal{O}(X)$. Then $g|A$ is a completely positive map which extends $f$, by Corollary \ref{19041501} our proof is complete.
\end{proof}

Combine Corollary \ref{190130} and Corollary \ref{19012801} we have

\begin{corollary}\label{hfdiusahufs}
If $C \subset A$ contains an approximate unit $\{b_i\}_{i \in I}$ of $A$ which satisfies $\|b_i\| \leq 1$ $(i \in I)$, then for any completely positive function $f: C \to \mathcal{O}(X)$ we have
\begin{equation*}
\|f\|_{\rm{cb}}={\rm{lim}}_{i \to \infty} \|f(b_i)\|
\end{equation*}
\end{corollary}
\begin{proof}
Let $\widetilde{f}$ be a $(S_{B,A})$-extension of $f$ on $M(A)$ (the existence is proved by Corollary \ref{19012801}), then by Corollary \ref{19041501} we have $\|\widetilde{f}\|_{\rm{cb}}=\|f\|_{\rm{cb}}$, our proof is complete.
\end{proof}

\bigskip

In the rest of this section, we will apply our definitions and propositions to $A= \mathcal{O}_c(\mathfrak{L}_2(G)) \otimes \rho_{{\rm{int}}}(C^{\ast}(\mathfrak{B}))$, and $C=\mathfrak{E}(\rho, \mathfrak{B})$. Recall that if $G$ is non-discrete, both $A$ and $C$ are separable $C^{\ast}$-algebra by Remark \ref{cvnhjfuk}. 
\begin{proposition}\label{19011801}
 Let $f: C \to \mathcal{O}(\mathfrak{L}_2(G,X(\rho)))$ be a completely bounded map. Consider the conditions:

(i) $f$ is $(S_{C,A})$-extendable;

(ii) $f$ has non-degenerate representation $(W,V,R)$ such that 
\begin{equation*}
R: C \to \mathfrak{E}(r, \mathfrak{B}), T_k^{\rho} \mapsto T_k^r \ \ \ \ (k \in \mathfrak{L}_2(\mathfrak{B} \times \mathfrak{B}))
\end{equation*}
for some $\ast$-representation $r$ of $\mathfrak{B}$ which is weakly contained in $\rho$. 

Then we have $(ii) \Rightarrow (i)$, and if either $G$ is discrete or $X(\rho)$ is separable we have $(i) \Rightarrow (ii)$. Furthermore, if $f$ is completely positive, we can choose $V=W$. 
\end{proposition}
\begin{proof}
$(i) \Rightarrow (ii)$ if $X(\rho)$ is separable:  We denote the extension of $f$ on $M(A)$ still by $f$. By Lemma \ref{19011601}, $f$ has non-degenerate representation $(V',W',S, Y)$,  such that $S|_A$ is non-degenerate $\ast$-representation of $A$, and by Lemma \ref{19011601} again,  $Y$ can be chosen to be separable.  By \cite[Lemma 2.5]{MTT16}, there is a separable Hilbert space $Z$, unitary operator $U: Y \to  \mathfrak{L}_2(G, Z)$ and non-degenerate $\ast$-representation $r': \rho_{{\rm{int}}}(C^{\ast}(\mathfrak{B})) \to \mathcal{O}(\mathfrak{L}_2(G,Z))$, such that
\begin{equation*}
S(T^{\rho_{\rm{int}}}_h)=U^{\ast} T^{r'\circ \rho_{\rm{int}}}_h U \ \ \ \  (h \in \mathfrak{L}_2(G \times G, C^{\ast}(\mathfrak{B})).
\end{equation*}
It is clear that $r' \circ \rho_{\rm{int}}: C^{\ast}(\mathfrak{B})\to \mathcal{O}(Z)$ is non-degenerate, thus $r' \circ \rho_{\rm{int}}$ is the integrated form of a $\ast$-representation of $\mathfrak{B}$ on $Z$ which we denote by $r: B \to \mathcal{O}(Z)$. Thus $r$ is weakly contained in $\rho$. Now for fixed $k \in \mathfrak{L}(\mathfrak{B} \times \mathfrak{B})$, by Lemma \ref{maintool} we have
\begin{equation*}\begin{split}
(U^{\ast}T_k^r U) \circ S(T_h^{\rho_{{\rm{int}}}})&=(U^{\ast}T_k^r U) \circ (U^{\ast} T_h ^{r' \circ \rho_{\rm{int}}} U)
\\&=U^{\ast} T_{k \star h}^{r' \circ \rho_{\rm{int}}} U  \ \ \ \ \ \ \  (h \in \mathfrak{K}({\mathfrak{B} })).
\end{split}\end{equation*}
On the other hand, by Lemma \ref{maintool} again, we have 
\begin{equation*}
S(T_k^{\rho}) S(T_h^{\rho_{{\rm{int}}}})=S(T_{k \star h}^{\rho_{{\rm{int}}}})=U^{\ast} T^{r' \circ \rho_{\rm{int}}}_{k \star h} U  \ \ \ \ (h \in \mathfrak{K}({\mathfrak{B} })).
\end{equation*}
Since $\mathfrak{K}({\mathfrak{B} })$ is dense in $A$ and $S|_A$ is non-degenerate $\ast$-representation of $A$, we conclude that
\begin{equation*}\label{ejdffgi}
S(T_k^{\rho})=U^{\ast}T_k^r U \ \ \ \ (k \in \mathfrak{L}(\mathfrak{B} \times \mathfrak{B})).
\end{equation*}
Now let $W=UW'$, $V=UV'$, and define $^{\ast}$-representation $R: \mathfrak{E}(\rho, \mathfrak{B}) \to \mathfrak{E}(r, \mathfrak{B})$ by $R(T_k^{\rho})=T_k^r$ ($k \in \mathfrak{L}_2(\mathfrak{B} \times \mathfrak{B})$) (see Lemma \ref{19011401}), since $\{T^{\rho}_k: k \in  \mathfrak{L}_2(\mathfrak{B} \times \mathfrak{B})\}$ is dense in $\mathfrak{E}(\rho, \mathfrak{B})$, by (\ref{ejdffgi}) we have
\begin{equation*}
f(a)=W^{\ast} R(a) V \ \ \ \ (a \in \mathfrak{E}(\rho, \mathfrak{B})).
\end{equation*}
$(i)  \Rightarrow (ii)$ is proved.

By the similar argument, we can prove that $(i) \Rightarrow (i)$ if $G$ is discrete.

$(ii) \Rightarrow (i)$: Since $r$ is weekly contained in $\rho$, then by Lemma \ref{19011401} the following map 
\begin{equation*}
S: T_h^{\rho_{{\rm{int}}}} \mapsto T_h^{r_{{\rm{int}}}} \ \ \ \ (h \in \mathfrak{L}_2(G \times G, C^{\ast}(\mathfrak{B})))
\end{equation*}
is $\ast$-homomorphism from $A$ onto $\mathcal{O}_c(\mathfrak{L}_2(G)) \otimes r_{{\rm{int}}}(C^{\ast}(\mathfrak{B}))$. Since $S$ is a non-degenerate, we define $F: M(A) \to \mathcal{O}(\mathfrak{L}_2(G,X))$ by
\begin{equation*}
F(a)=W^{\ast} S(a) V \ \ \ \ (a \in M(A)).
\end{equation*}
By Lemma \ref{19011601} $F$ is $(S_A)$-map on $M(A)$, and $F|A$ is extension of $f$. 
\end{proof}

For any $a \in \mathfrak{L}_{\infty}(G)$, we define $\overline{a} \in \mathfrak{L}_{\infty}(G)$ by $\overline{a}(s)=\overline{a(s)}$ for all $s \in G$. 

For each $a \in \mathfrak{L}_{\infty}(G)$ we associate an operator $M_a: \mathfrak{L}_2(G) \to \mathfrak{L}_2(G)$ which is defined by
\begin{equation}
M_a(\xi)(s)=a(s) \xi(s) \ \ \ \ (\xi \in \mathfrak{L}_2(G); s \in G ),
\end{equation}
it is clear that $M_a^{\ast}=M_{\overline{a}}$, thus $\mathfrak{M}(G)=\{M_a: a \in \mathfrak{L}_{\infty}(G)\}$ is a concrete $C^{\ast}$-algebra acting on Hilbert space $\mathfrak{L}_2(G)$. 

 For $k \in \mathfrak{L}_2(\mathfrak{B} \times \mathfrak{B})$ we define
\begin{equation*}
a \cdot k(s,t)=a(s) k(s,t), k \cdot a (s,t)=k(s,t) a(t)
\end{equation*}
for $a \in \mathfrak{L}_{\infty}(G)$ and $s,t \in G$. Then it is easy to verify that $a\cdot k$, $k \cdot a$ are in $\mathfrak{L}_2(\mathfrak{B} )$ and For any $k \in \mathfrak{L}_2(\mathfrak{B} \times \mathfrak{B})$ and $a,b \in \mathfrak{L}_{\infty}(G)$, 
\begin{equation*}
(M_a \otimes I_X) T_k^{\rho} (M_b \otimes I_X)=T^{\rho}_{a \cdot k \cdot b}, 
\end{equation*}
and if $\Phi$ is Schur $(\mathfrak{B} , \rho)$-multiplier, we have (recall that $X$ is the Hilbert space on which $\rho$ is acting)
\begin{equation*}
S_{\rho,\Phi}((M_a \otimes I_X) x (M_b \otimes X))=(M_a \otimes I_X) S_{\rho,\Phi}(x) (M_a \otimes I_X)
\end{equation*}
for all $  \in \mathfrak{E}(\rho, \mathfrak{B} ) $ and $a,b \in \mathfrak{L}_{\infty}(G)$.

\begin{theorem}\label{19012502}
Let $\Phi: D \to D$ be a multiplier of $\mathfrak{B} \times \mathfrak{B}$, and consider the following two statements:

(i) $\Phi$ is  Schur $(\mathfrak{B} , \rho)$-multiplier such that $S_{\rho,\Phi}$ is $(S_{C,A})$-extendable;

(ii) There is non-degenerate $\ast$-representation $r: B \to \mathcal{O}(Y)$ of $\mathfrak{B}$ on some Hilbert space $Y$ which is weakly contained in $\rho$, and $V, W \in \mathfrak{L}_{\infty}(G, \mathcal{O}(X, Y))$ such that
\begin{equation*}
\rho(\Phi((x,y; a))= W^{\ast}(x) r_D((x,y; a)) V(y) \ \ \ \ ((x,y; a) \in D_{x,y})
\end{equation*}
for almost all $(x,y) \in G \times G$.
Then we have $(ii) \Rightarrow (i)$, and if either $X(\rho)$ is separable or $G$ is discrete then we have $(i) \Rightarrow (ii)$.
\end{theorem}
(If $G$ is discrete, then by \cite[VIII.16.11]{MR936629} for any ${\ast}$-representation $\rho$ of $\mathfrak{B}$ we always have $\rho(B) \subset \rho_{{\rm{int}}}(C^{\ast}(\mathfrak{B}))$, thus we have
\begin{equation*}
 \mathfrak{E}(\rho, \mathfrak{B} ) \subset \mathcal{O}_c(\mathfrak{L}_2(G)) \otimes \rho_{\rm{int}}(C^{\ast}(\mathfrak{B})), 
 \end{equation*}
 Corollary \ref{19012801} implies that every completely bounded multipliers is 
 $(S_{C,A})$-extendable. Therefore, we can remove `$S_{\rho,\Phi}$ is $(S_{C,A})$-extendable' in $(i)$, and (i) and $(ii)$ are always equivalent.)

\begin{proof}
We prove that $(i) \Rightarrow (ii)$ if $X(\rho)$ is separable, the argument in the case that $G$ is discrete is similar. Since $S_{\rho,\Phi}$ is $(S_{C,A})$-extendable, by Proposition \ref{19011801} there are separable Hilbert space $Y$, non-degenerate $\ast$-representation $r: B \to \mathcal{O}(Y)$ which is weakly contained in $\rho$, and bounded operators $V, W$  in $\mathcal{O}(\mathfrak{L}_2(G, X), \mathfrak{L}_2(G,Y))$ such that
\begin{equation*}
{S_{\rho}},_{\Phi}(T^{\rho}_k)=W_0^{\ast} T_k^r V_0\ \ \ \ ( k \in \mathfrak{L}_2(\mathfrak{B} \times \mathfrak{B})).
\end{equation*}
Let $C=\overline{[\{T^r_k V_0 \mathfrak{L}_2(G,X): k \in \mathfrak{L}_2(\mathfrak{B} \times \mathfrak{B})\}]} \subset \mathfrak{L}_2(G, Y)$, and let $E$ be the prjection of $\mathfrak{L}_2(G,Y)$ onto $C$, then since $C$ is invariant under the concrete $C^{\ast}$-algebra $\mathfrak{E}(r, \mathfrak{B} )$ , we have $ET_k^r=T_k^rE$ for all $k \in \mathfrak{L}_2(\mathfrak{B} )$. Now for any $d \in \mathfrak{L}_{\infty}(G)$ and $k \in \mathfrak{L}_2(\mathfrak{B})$ we have
\begin{equation}\label{19041601}
(M_d \otimes I_Y)^{\ast}T_k^r=T_{\overline{d} \cdot k}^r \ \ \ \ (d \in \mathfrak{L}_{\infty}(G); k \in \mathfrak{L}_2(\mathfrak{B} \times \mathfrak{B}))
\end{equation}
and
\begin{equation}\begin{split}\label{19011802}
W_0^{\ast}T_{\overline{d} \cdot k} V_0=(M_d^{\ast} \otimes I_X) \ W	^{\ast}_0 T^r_k V_0.
\end{split}\end{equation}
Let $W=E W_0$, combine (\ref{19011802}) and (\ref{19041601}) we have
\begin{equation*}\begin{split}
\langle T_k^r V_0 \xi, (M_d \otimes I_Y) W \eta \rangle=\langle T_k^r V_0\xi, W(M_d \otimes I_X) \eta  \rangle
\end{split}\end{equation*}
for all $k  \in \mathfrak{L}_2(\mathfrak{B})$ and $\xi, \eta \in \mathfrak{L}_2(G, X)$. We conclude that
\begin{equation*}
E(M_d \otimes I_Y)W=EW(M_d \otimes I_X) \ \ \ \ (d \in \mathfrak{L}_{\infty}(G)).
\end{equation*}
Furthermore,  $C$ is stable under the action of the $C^{\ast}$-algebra $\{M_d \otimes I_Y: d \in \mathfrak{L}_{\infty}(G)\}$, thus we have
\begin{equation*}
E(M_d \otimes I_Y)=(M_d \otimes I_Y) E \ \ \ \ (d \in \mathfrak{L}_{\infty}(G)),
\end{equation*}
we conclude that
\begin{equation*}
(M_d \otimes I_Y)W=W(M_d \otimes I_X) \ \ \ \ (d \in \mathfrak{L}_{\infty}(G)),
\end{equation*}
it follows that $W \in \mathfrak{L}_{\infty}(G, \mathcal{O}(X,Y))$. 
By the same argument we can replace $V_0$ by an operator $V \in \mathfrak{L}_{\infty}(G, \mathcal{O}(X,Y))$, and 
by the same argument of the proof of \cite[Theorem 2.6]{MTT16} we have for any $k \in \mathfrak{L}_2(\mathfrak{B} \times \mathfrak{B})$
\begin{equation*}
\rho_D(\Phi(k(x,y)))=W(x)^{\ast}r_D(k(x,y))V(y) \ \ \ \ (k \in \mathfrak{L}_2(\mathfrak{B} \times \mathfrak{B}))
\end{equation*}
for almost all $(x,y) \in G\times G$.

Now since $\mathfrak{B} \times \mathfrak{B}$ is second-countable Banach bundle over second countable group $G$, $\mathfrak{L}(\mathfrak{B} \times \mathfrak{B})$ is separable. Let $\{k_n\}_{n \in \mathbb{N}} \subset \mathfrak{L}(\mathfrak{B} \times \mathfrak{B})$ be the countable dense subset of $\mathfrak{L}(\mathfrak{B} \times \mathfrak{B})$. In particular, for fixed $x,y \in G$ the set $\{k_n(x,y): n \in \mathbb{N}\}$ is dense in $D_{x,y}$. On the other hand, for each $n \in \mathbb{N}$ we define
\begin{equation*}
N_n=\{(x,y) \in G \times G: W(x)^{\ast} r_D(k_n(x,y))V(y) \neq \rho_D(\Phi(k_n(x,y)))\},
\end{equation*}
then $N=\bigcup_{n \in \mathbb{N}} N_n$ is null-subset of $G \times G$. If $(x,y) \in (G \times G) \setminus N$, then for any $a_{x,y} \in D_{x,y}$ we may have a sequence $\{k_{n_i}\}_{i \in \mathbb{N}} \subset \{k_n\}_{n \in \mathbb{N}}$ such that $k_{n_i}(x,y) \to a_{x,y}$, hence we have
\begin{equation*}\begin{split}
\rho_D(\Phi(a_{x,y}))=W(x)^{\ast} r_D(a_{x,y}) V(y).
\end{split}\end{equation*}
The proof of $(i) \Rightarrow (ii)$ is complete.

\bigskip
$(ii) \Rightarrow (i):$ For any $k \in \mathfrak{L}_2(\mathfrak{B} \times \mathfrak{B})$ we have
\begin{equation*}\begin{split}
\langle S_{\rho,\Phi}(T_k^{\rho}) \xi, \eta \rangle&= \int_G  W(x)^{\ast} r_D(k(x,y)) V(y) \xi(y), \eta(x) \rangle dx dy
\\&=\langle(W^{\ast}T_k^{\ast} V) \xi, \eta  \rangle \ \ \ \ (\xi, \eta \in \mathfrak{L}_2(G,X)),
\end{split}\end{equation*}
thus we have
\begin{equation*}
S_{\rho,\Phi}(T_k^{\rho})=W^{\ast}T_k^r V. 
\end{equation*}
Since $r$ is weakly contained in $\rho$, by Lemma \ref{19011401} $S_{\rho,\Phi}$ is completely bounded.
\end{proof}

\section{Herz-Schur multipliers}

In this section, we fix $\rho: B \to \mathcal{O}(X)$ to be a non-degenerate $\ast$-representation of Fell bundle $\mathfrak{B}$.

Let $\Psi: B \to B$ be a multiplier of $\mathfrak{B}$, for each $f \in \mathfrak{L}(\mathfrak{B})$ we define $\Psi \cdot f \in \mathfrak{L}(\mathfrak{B})$ by
\begin{equation*}
\Psi \cdot f(x)=\Psi(f(x)) \ \ \ \ (x \in G).
\end{equation*}

\begin{definition}
Let $\varsigma: B \to \mathcal{O}(Y)$ be a $\ast$-representation of $\mathfrak{B}$ on Hilbert space $Y$ such that $\varsigma|B_e$ is faithful. A  multiplier $\Psi: B \to B$ of $\mathfrak{B}$ is called $(\varsigma, \mathfrak{B})$-multiplier if the map $S_{\Psi}^{\varsigma}$ which is defined on $\{\varsigma_{{\rm{int}}}(f): f \in \mathfrak{L}_1(\mathfrak{B})\} \subset \mathcal{O}(Y)$ by
\begin{equation*}
S_{\Psi}^{\varsigma}(\varsigma_{{\rm{int}}}(f))=\varsigma_{{\rm{int}}}(\Psi \cdot f) \ \ \ \ ( f \in \mathfrak{L}_1(\mathfrak{B}))
\end{equation*}
is completely bounded. In this case, $S_{\Psi}^{\varsigma}$ may be extended to a completely bounded map on $\{\varsigma_{{\rm{int}}}(f):  f \in \mathfrak{L}_1(\mathfrak{B})\}^{-\mathcal{O}(Y)}$ which we still denote by $S_{\Psi}^{\varsigma}$.
\end{definition}
\begin{remark}\label{xcvhduk}
Notice that if $r$ and $\varsigma$ are weakly equivalent $^{\ast}$-representation of $\mathfrak{B}$, then it is easy to see that $\Psi$ is $(\varsigma, \mathfrak{B})$-multiplier if and only if $\Psi$ is $(r, \mathfrak{B})$-multiplier, and in this case
\begin{equation*}
\|S_{\Psi}^{\varsigma}\|_{\rm{cb}}=\|S_{\Psi}^{r}\|_{\rm{cb}}.
\end{equation*}
\end{remark}

We use symbol $\lambda_{\mathfrak{B}}$ to denote the regular $^{\ast}$-representation of $C^{\ast}(\mathfrak{B})$, and we denote the reduced $C^{\ast}$-algebra of $\mathfrak{B}$, i.e $C^{\ast}(\mathfrak{B}) / {\rm{Ker}}(({\lambda_{\mathfrak{B}}})_{\rm{int}})$, by $C^{\ast}_{\rm{Red}}(\mathfrak{B})$ (see \cite{MR1891686}).

 Let $\lambda: G \to \mathcal{O}(\mathfrak{L}_2(G))$ be the left regular representation of $G$. By \cite[VIII.9.16]{MR936629}, we can form $^{\ast}$-representation $\rho \otimes \lambda$ of $\mathfrak{B}$ on $\mathfrak{L}_2(G) \otimes X(\rho)$ defined by
 \begin{equation*}
 (\rho \otimes \lambda)_b=\lambda_{\pi(b)} \otimes \rho_b \ \ \ \ (b \in B).
 \end{equation*}
 From \cite{MR1891686} we can identify $C_{{\rm{Red}}}^{\ast}(\mathfrak{B})$ with
\begin{equation*}
\{(\rho \otimes \lambda)_{{\rm{int}}}(f): f \in \mathfrak{L}_1(\mathfrak{B})\}^{-\mathcal{O}(\mathfrak{L}_2(G,X))}.
\end{equation*}
For any $f \in \mathfrak{L}(\mathfrak{B})$ and $\xi \in \mathfrak{L}_2(G,X)$ we have
\begin{equation*}
(\rho \otimes \lambda)_{{\rm{int}}}(f)(\xi)(x)=\int_G \rho(f(y))(\xi(y^{-1}x)) dy
\end{equation*}
for almost all $x \in G$.

\begin{definition}\label{dfhjfv}
We call multiplier $\Psi: B \to B$ of $\mathfrak{B}$ Herz-Schur multiplier of $\mathfrak{B}$ if $\Psi$ is $(\rho \otimes \lambda, \mathfrak{B})$-multiplier. Furthermore, if $S_{\Psi}^{\rho \otimes \lambda}$ is completely positive we call $\Psi$ completely positive Herz-Schur multiplier of $\mathfrak{B}$. 

If $\Psi$ is Herz-Schur multiplier, we define

\begin{equation*}
\|\Psi\|_{\mathfrak{H.S}}=\|S^{\rho \otimes \lambda}_{\Psi}\|_{\rm{cb}}.
\end{equation*}
\end{definition}
\begin{remark}
By Remark \ref{xcvhduk}, Definition \ref{dfhjfv} is independent on the choice of $^{\ast}$-representation of $\rho: B \to \mathcal{O}(X(\rho))$ for which $\rho|_{B_e}$ is faithful. 
\end{remark}

Therefore if $\Psi$ is Herz-Schur multiplier, $S_{\Psi}^{\rho \otimes \lambda}$ may be extended to $\{(\rho \otimes \lambda)_{{\rm{int}}}(f): f \in \mathfrak{L}_1(\mathfrak{B})\}^{-\mathcal{O}(\mathfrak{L}_2(G,X))}$, thus we may regard $S_{\Psi}$ is completely bounded map defined on $C^{\ast}_{{\rm{Red}}}(\mathfrak{B})$. 

In the following, if $\Psi$ is Herz-Schur multiplier, we will denote $S_{\Psi}^{\rho \otimes \lambda}$ briefly by $S_{\Psi}$.

Th following lemma is well-known, we list it here for the concenience for reference
\begin{lemma}
Let $A$ be $C^{\ast}$-algebra and $Y$ a Hilbert space. If $f: A \to \mathcal{O}(Y)$ is completely bounded (resp. positive) map, then there is a non-degenerate representation $(V,W, R, Z)$ $(resp.(V,R,Z))$ of $f$ such that $\|f\|_{\rm{cb}}=\|V\|\|W\|$ (resp. $\|V\|^2$).
\end{lemma}

\begin{proposition}\label{19012301}
Let $\varsigma: B \to \mathcal{O}(Y)$ be a $\ast$-representation of $\mathfrak{B}$, $\Psi: B \to B$ be a multiplier of $\mathfrak{B}$, the following are equivalent:

(i) $\Psi$ is $(\varsigma, \mathfrak{B})$-multiplier;

(ii) There is non-degenerate $\ast$-representation $r: B \to \mathcal{O}(Z)$ on Hilbert space $Z$ which is weakly contained in $\varsigma$, and bounded operators $V,W \in \mathcal{O}(Y,Z)$ $($here $Y=X(\varsigma)$$)$ such that
\begin{equation*}
\varsigma(\Psi(a_x))=W^{\ast}r(a_x)V \ \ \ \ (a_x \in B_x)
\end{equation*}
for all $x \in G$ and $\|S^{\varsigma}_{\Psi}\|_{\rm{cb}}=\|V\|\|W\|$.
\end{proposition}
\begin{proof}
$(i) \Rightarrow (ii):$ Since $S_{\Psi}^{\varsigma}: \varsigma_{{\rm{int}}}(C^{\ast}(\mathfrak{B})) \to \mathcal{O}(Z)$ is completely bounded map, let $(W, V, R, Z)$ be its representation.
Therefore, let $r: B \to \mathcal{O}(Z)$ be $\ast$-representation of $\mathfrak{B}$ whose integrated form is $R \circ \varsigma_{\rm{int}}$, we have
\begin{equation*}\begin{split}
S_{\Psi}^{\varsigma}(\varsigma_{{\rm{int}}}(f))&=W^{\ast}r_{{\rm{int}}}(f)V  
\\&=W^{\ast} \ \int_G r(f(x)) dx \ V \ \ \ \ (f \in \mathfrak{L}_1(\mathfrak{B})).
\end{split}\end{equation*}
Now let $y \in G$ and $a_y \in B_y$, let $\{g_i\}_{i \in I} \subset \mathfrak{L}(G)$ be a net such that ${\rm{supp}}(g_i) \to y$, $g_i \geq 0$ and $\int_G g_i(x) dx=1$. Furthermore, let $f \in \mathfrak{L}(\mathfrak{B})$ such that $f(y)=a_y$, we have
\begin{equation*}\begin{split}
S_{\Psi}^{\varsigma}(\varsigma_{{\rm{int}}}(g_i f)) \underset{{\rm{strong \ operator}}}{\longrightarrow} \varsigma(\Psi(f(y))).
=\varsigma(\Psi(a_y))
\end{split}\end{equation*}
On the other hand
\begin{equation*}\begin{split}
W^{\ast} \left(\int_G \varsigma(g_i(x)f(x)) dx \right) V  & \underset{{\rm{strong \ operator}}}{\longrightarrow} W^{\ast} r(f(y)) V
\\&=W^{\ast} r(a_y) V.
\end{split}\end{equation*}
Therefore we proved that
\begin{equation*}
\varsigma(\Psi(a_y))=W^{\ast} r(a_y) V \ \ \ \ (y \in G; a_y \in B_y)
\end{equation*}

$(ii) \Rightarrow (i):$ Since $r$ is weakly contained in $\varsigma$, then we have $\ast$-homomorphism $R: \varsigma_{{\rm{int}}}(C^{\ast}(\mathfrak{B})) \to r_{{\rm{int}}}(C^{\ast}(\mathfrak{B}))$ which satisfies
\begin{equation*}
R(\varsigma_{{\rm{int}}}(f))=r_{{\rm{int}}}(f) \ \ \ \ (f \in \mathfrak{L}_1(\mathfrak{B})).
\end{equation*}
Therefore for arbitrary $\xi \in Y$ we have
\begin{equation*}\begin{split}
\varsigma_{{\rm{int}}}(\Psi \cdot f)(\xi)=W^{\ast}R(\varsigma_{{\rm{int}}}(f)) V(\xi) \ \ \ \ (f \in \mathfrak{L}_1(\mathfrak{B})),
\end{split}\end{equation*}
thus $\varsigma_{{\rm{int}}}(\Psi \cdot f)=W^{\ast}R(\varsigma_{{\rm{int}}}(f)) V$ for all $f \in \mathfrak{L}_1(\mathfrak{B})$, we may conclude that $S_{\Psi}^{\varsigma}$ is completely bounded.
\end{proof}

Let $\Psi: B \to B$ be a multiplier of $\mathfrak{B}$, then we define a map $\mathfrak{B} \times \mathfrak{B}$ $N(\Psi): D  \to D $ by
\begin{equation}\label{cvjdfukfwuriue}
N(\Psi)((x,y; a))=(x,y; \Psi(a)) \ \ \ \ (x,y \in G; (x,y; a) \in D_{x,y}).
\end{equation}
It is routine to check that $N(\Psi)$ satisfies $(i)$-$(ii)$ of Definition \ref{chsduiiusrfcnh}, thus $N(\Psi)$ is multiplier of $\mathfrak{B} \times \mathfrak{B}$.

\begin{theorem}\label{19012501}
Suppose that $\rho$ weakly contains $\lambda_{\mathfrak{B}}$. Let $\Psi: B \to B$ be a multiplier of $\mathfrak{B}$. The following are equivalent:

(i) $\Psi$ is (resp. completely positive) Herz-Schur multiplier of $\mathfrak{B}$;

(ii) $N(\Psi)$ is (resp. completely positive) Schur $(\rho,\mathfrak{B})$-multiplier.

If either $(i)$ or $(ii)$ holds, we have
\begin{equation*}
\|\Psi\|_{\mathfrak{H.S}}=\|N(\Psi)\|_{\mathfrak{S}}.
\end{equation*}
\end{theorem}
\begin{proof}
$(i) \Rightarrow (ii):$  By Proposition \ref{19012301} we have $\ast$-representation $r: B \to \mathcal{O}(Y)$ on Hilbert space $Y$ which is weakly contained in $\lambda_{\mathfrak{B}}$ and bounded operators $W_0,V_0 \in \mathcal{O}(X,Y)$ such that
\begin{equation*}\begin{split}
\rho \otimes \lambda (\Psi(a_x))=W_0^{\ast} r(a_x) V_0 \ \ \ \ (x \in G; a_x \in B_x).
\end{split}\end{equation*}
By Proposition \ref{19012502} and the same argument of $(i) \Rightarrow (ii)$ it is not hard to prove that $N(\Psi)$ is Schur $(\rho \otimes 1_{\mathcal{O}(\mathfrak{L}(G))}, \mathfrak{B})$-multiplier, such that 
\begin{equation*}
\|S_{\rho \otimes 1_{\mathcal{O}(\mathfrak{L}_2(G))},N(\Psi)}\|_{\rm{cb}} \leq \|V\|\|W\|=\|S_{\Psi}\|_{\rm{cb}}. 
\end{equation*}
By Proposition \ref{19012303} $N(\Psi)$ is $(\rho, \mathfrak{B})$-multiplier.

$(ii) \Rightarrow (i)$: By the same argument of $(ii) \Rightarrow (i)$ of the proof of \cite[Theorem 3.8]{MTT16} $\Psi$ is Herz-Schur multiplier and 
\begin{equation*}
\|(\rho \otimes \lambda)_{{\rm{int}}}(\Psi \cdot f) \|_{\rm{cb}} \leq \|S_{N(\Psi)}\|_{\rm{cb}}.
\end{equation*}
Thus $\|(\rho \otimes \lambda)_{{\rm{int}}}(\Psi \cdot f) \|_{\rm{cb}}=\|S_{N(\Psi)}\|_{\rm{cb}}$.
The proof is complete.
\end{proof}

\begin{corollary}
If $\varsigma: B \to \mathcal{O}(Y)$ is $\ast$-representation of $\mathfrak{B}$ which weakly contains $\lambda_{\mathfrak{B}}$, then any $(\varsigma, \mathfrak{B})$-multiplier is Herz-Schur multiplier.
\end{corollary}
\begin{proof}
This is the combination of Theorem \ref{19012501}, Proposition \ref{19012301} and Theorem \ref{19012502}.
\end{proof}

\section{Nuclearifty of Cross-Sectional Algebra}

In this section we assume that $\mathfrak{B}$ is a Fell-bundle over discrete group $G$ and that  $\rho$ is a fixed $^{\ast}$-representation of $\mathfrak{B}$ which weakly contains $\lambda_{\mathfrak{B}}$. Recall that  in this case we always have
\begin{equation*}
\mathfrak{E}(\rho, \mathfrak{B}) \subset \mathcal{O}_c(\mathfrak{L}_2(G)) \otimes \rho_{{\rm{int}}}(C^{\ast}(\mathfrak{B})).
\end{equation*}
Therefore by Lemma \ref{19012801} and Theorem \ref{19012502} we have
\begin{corollary}\label{19021202}
Let $\Phi: D \to D$ be a multiplier of $\mathfrak{B} \times \mathfrak{B}$, the the following are equivalent:

(i) $\Phi$ is (resp. completely positive) Schur $(\mathfrak{B} , \rho)$-multiplier.

(ii) There are non-degenerate $\ast$-representation $r: B \to \mathcal{O}(Y)$ on some Hilbert space $Y$ which is weakly contained in $\rho$, and $V, W \in \mathfrak{L}_{\infty}(G, \mathcal{O}(X, Y))$ (resp.$\,V=W$) such that
\begin{equation*}
\rho_D(\Phi((x,y; a))= W^{\ast}(x) r(a) V(y) \ \ \ \ (x,y \in G; \ a \in B_{xy^{-1}}).
\end{equation*}
\end{corollary}

\bigskip

Let $\Phi : D \to D$ be a multiplier of $\mathfrak{B} \times \mathfrak{B}$, we define $\Phi^{\rho}_{x}: \rho_D(D_{x,x}) \to \rho_D(D_{x,x})$ ($x \in G$) by
\begin{equation*}
\Phi^{\rho}_x(\rho_D((x,x; a))=\rho_D(\Phi((x,x; a)) \ \ \ \ ((x,x; a) \in D_{x,x}). 
\end{equation*}

\begin{remark}
Let us recall that each $D_{x,x}=B_e$ $(x \in G)$. However it is not necessary that all $\Phi_x^{\rho}$ are identical. 
\end{remark}

\begin{proposition}\label{19012901}
Let $\Phi:  D \to D$ be a completely positive Schur $(\rho, \mathfrak{B})$-multiplier, then each $\Phi^{\rho}_x$ is completely positive and
\begin{equation*}
\|\Phi\|_{\mathfrak{S}}={\rm{sup}}\{\|\Phi_x^{\rho}\|_{\rm{cb}}: x \in G\}. 
\end{equation*}
\end{proposition}
\begin{proof}
That each $\Phi^{\rho}_x$ is completely positive is by Corollary \ref{19021202}.

Let $\{a_i\}_{i \in I}$ be an approximate unit of the $C^{\ast}$-algebra $B_e$ with $\|a_i\| \leq 1$ $(i \in I)$. Let $\mathcal{A}$ be the collection of all finite subsets of $G$.  We define the order on $\mathcal{A} \times I$ by
\begin{equation*}
(U,i) \leq (V,j) \Leftrightarrow U \subset V \ {\rm{and}} \ i \leq j \ \ \ \ ((U,i), (V,j) \in \mathcal{A} \times I).
\end{equation*}
Then let define $a_{U,i}=a_i$ for all $(U,i) \in \mathcal{A} \times I$, it is clear that $\{a_{U,i}\}_{(U,i) \in \mathcal{A} \times I}$ is an approximate unit of $B_e$ since it is a subnet of $\{a_i\}_{i \in I}$.  By \cite[VIII.5.11]{MR936629} and \cite[VIII.16.3]{MR936629}, we have that $\|a_{U,i}b-b\| \to 0$ for all $b \in C^{\ast}(\mathfrak{B})$. 

Let define $k_{U,i} \in \mathfrak{L}(\mathfrak{B} \times \mathfrak{B})$ by
\begin{equation*}\begin{split}
&k_{(U,i)}(x,x)=(x,x; a_{U,i}) \ \ \ \ (x \in U);
\\&  k_{(U,i)}(y,z)=0 \ \ \ \ {\rm{otherwise}},
\end{split}\end{equation*}
thus if $(U,i) \to \infty$ we have
\begin{equation}\begin{split}\label{cbnhifuyeruis}
{\rm{lim}}_{(U,i) \to \infty}T^{\rho_{{\rm{int}}}}_{k_{U,i}} T_k^{\rho}=T_k^{\rho_{\rm{int}}} \ \ \ \ (k \in \mathfrak{L}(\mathfrak{B} \times \mathfrak{B})).
\end{split}\end{equation}
Since $\{T_k^{\rho_{{\rm{int}}}}: k \in \mathfrak{L}(G \times G, C^{\ast}(\mathfrak{B}))\}$ is dense in $\mathcal{O}_c(\mathfrak{L}_2(G)) \otimes \rho_{{\rm{int}}}(C^{\ast}(\mathfrak{B}))$, (\ref{cbnhifuyeruis}) implies that $\{T^{\rho}_{k_{U,i}}\}$ is an approximate unit of $\mathcal{O}_c(\mathfrak{L}_2(G)) \otimes \rho_{{\rm{int}}}(C^{\ast}(\mathfrak{B}))$ with norm not greater than 1.

 On the other hand, for each fixed $x \in G$, since $\rho|B_e$ is non-degenerate $\ast$-representation of $B_e$, we conclude that $\|\rho_D(k_{U,i}(x,x))\| \leq 1$ and that the net 
 \begin{equation*}
 \{\rho_D(k_{U,i}(x,x))\}_{(U,i) \in \mathcal{A} \times I}
\end{equation*} 
is approximate unit of the concrete $C^{\ast}$-algebra $\rho_D(D_{x,x})=\rho(B_e) \subset \mathcal{O}(X)$ . By Corollary \ref{hfdiusahufs} we have
\begin{equation*}\begin{split}
\|\Phi\|_{\mathfrak{S}}&={\rm{lim}}_{(U,i) \to \infty} \|S_{\rho,\Phi}(T^{\rho}_{k_{(U,i)}})\|
\\&={\rm{sup}} \{\|\Phi^{\rho}_x\|: x \in G\}.
\end{split}\end{equation*}
\end{proof}

For multiplier $\Psi: B \to B$ of $\mathfrak{B}$ we define $\Psi^{\rho}_x: \rho(B_x) \to \rho(B_x)$ by $\Psi^{\rho}_x(\rho(a_x))=\rho(\Psi(a_x))$ $(a_x \in B_x)$ for all $x \in G$.

\begin{corollary}\label{19021301}
If $\Psi: B \to B$ is completely positive Herz-Schur multiplier of $\mathfrak{B}$ and $\rho$ weakly contain $\lambda_{\mathfrak{B}}$, then we have
\begin{equation*}
\|\Psi\|_{\mathfrak{H.S}}=\|\Psi^{\rho}_e\|_{\rm{cb}}.
\end{equation*}
\end{corollary}

\begin{definition}\label{vnnfduorfjlgf}
We call Fell bundle $\mathfrak{B}$ nuclear if there exists a net $\{\Psi\}_{i \in I}$ of completely positive Herz-Schur multipliers of $\mathfrak{B}$ such that

i. $\|\Psi_e^{\rho}\|_{\rm{cb}} \leq 1 $ for all $i \in I$;

ii. Each $(\Psi_i)_x^{\rho}$ has finite dimensional range $(x \in G; i \in I)$; 

iii.$\|(\Psi_i)_{x}^{\rho}(\rho(a_x))-\rho(a_x)\| \to 0$ provided $i \to \infty$ for all $s \in G$ and $a_x \in B_x$.
\end{definition}

If $A$ is nuclear $C^{\ast}$-algebra, and $\{\phi_i\}_{i \in I}$ is a net of completely positive contractive maps on $A$ such that the range of each $\phi_i$ is finite and $\|\phi_i(a)-a\| \to 0$ for all $a \in A$ provided $i \to \infty$, we call $\{\phi_i\}_{i \in I}$ is an approximation net of $A$.

\bigskip

Recall that $C_{\rm{Red}}^{\ast}(\mathfrak{B})$ is $^{\ast}$-isomorphic to $(\rho \otimes \lambda)_{\rm{int}}(C^{\ast}(\mathfrak{B}))$. In the following we choose $^{\ast}$-representation $\rho: B \to \mathcal{O}(X(\rho))$ to be such that $\rho$ is weakly equivalent to $\lambda_{\mathfrak{B}}$. By \cite[Proposition 19.3]{MR1891686}, there is a conditional expectation 
\begin{equation*}
\mathcal{E}: \rho_{\rm{int}}(C^{\ast}(\mathfrak{B})) \otimes_{\rm{min}} C_r^{\ast}(G) \to \rho_{\rm{int}}(C^{\ast}(\mathfrak{B}))
\end{equation*}
satisfies 
\begin{equation*}
\mathcal{E}(\sum_{x \in G} \rho_{\rm{int}}(a) \otimes \lambda_x))=\rho_{\rm{int}}(a) \ \ \ \ (a \in C^{\ast}(\mathfrak{B})).
\end{equation*}
We have
\begin{equation*}
\mathcal{E}(a  (1_{\mathcal{O}(X(\rho))} \otimes \lambda_y))=\mathcal{E}((1_{\mathcal{O}(X(\rho))} \otimes \lambda_y) a)
\end{equation*}
for $a \in \rho_{\rm{int}}(C^{\ast}(\mathfrak{B})) \otimes_{\rm{min}} C_r^{\ast}(G), \ x,y \in G$.

\begin{remark}\label{uryugndfkfus}
Take a note that $C_{\rm{Red}}^{\ast}(\mathfrak{B})=(\rho \otimes \lambda)_{\rm{int}}(C^{\ast}(\mathfrak{B}))$ is a $C^{\ast}$-subalgebra of $\rho_{\rm{int}}(C^{\ast}(\mathfrak{B})) \otimes_{\rm{min}} C_r^{\ast}(G)$, for we have
\begin{equation*}
\sum_{x \in G} \rho(f(x)) \otimes \lambda_x \in \rho_{\rm{int}}(C^{\ast}(\mathfrak{B})) \otimes_{\rm{min}} C_r^{\ast}(G) \ \ \ \ (f \in \mathfrak{L}(\mathfrak{B})).
\end{equation*}
\end{remark}

Let $F: C_{\rm{Red}}^{\ast}(\mathfrak{B}) \to  C_{\rm{Red}}^{\ast}(\mathfrak{B})$ be a completely positive map, then we could regard $F$ as a completely positive map on $(\rho \otimes \lambda)_{\rm{int}}(C^{\ast}(\mathfrak{B}))$.  Let $(V,r)$ be a representation of $F$, where $r: \rho_{\rm{int}}(C^{\ast}(\mathfrak{B})) \to X(r)$, $V \in \mathcal{O}(\mathfrak{L}_2(G,X(\rho)), X(r))$ (notice that $\rho$ weakly contains $r$). Since $\rho|_{B_x}$ is faithful for each $x \in G$, we identify $B_x$ with $\rho(B_x)$ in the following. Since it is easy to verify that $\mathcal{E}(F(\rho(a_x) \otimes \lambda_x) (1_{\mathcal{O}(X(\rho))} \otimes \lambda_x)^{\ast})  \in B_x$, we can define the multiplier $h_F: B \to B$ by
\begin{equation*}
h_F(a_x)=\mathcal{E}(F(\rho(a_x) \otimes \lambda_x) (1_{\mathcal{O}(X(\rho))} \otimes \lambda_x)^{\ast}) \ \ \ \ (a_x \in B_x).
\end{equation*}

\begin{proposition}\label{19021402}
$h_F$ is completely positive Herz-Schur multiplier such that $\|h_F\|_{\mathfrak{H}.\mathfrak{S}} \leq \|F\|_{\rm{cb}}$.
\end{proposition}
\begin{proof}
Since $\rho|_{B_e}$ is faithful, we identify $B_e$ with $\rho(B_e)$. We have
\begin{equation*}\begin{split}
\rho_D(N(h_F)((x,y; a)))=\mathcal{E}(((1_{\mathcal{O}(X(\rho))} \otimes \lambda_x) V^{\ast}) r(a) (V (1_{\mathcal{O}(X(\rho))} \otimes \lambda_y^{\ast})))
\end{split}\end{equation*}
for all $a \in B_{xy^{-1}}$. Since $\mathcal{E}$ is completely positive, and $r$ is weakly contained in $\rho$ and that $\mathcal{E}$ is completely positive $x \mapsto V (1_{\mathcal{O}(X(\rho))} \otimes \lambda_x^{\ast})$ is bounded, then by Theorem \ref{19021202} $N(h_F)$ is Schur $\rho$-multiplier. By Theorem \ref{19012501} $h_F$ is completely positive Herz-Schur multiplier.
 
 Furthermore by \ref{19021301} we have
\begin{equation*}\begin{split}
\|h_F\|_{\mathfrak{H}.\mathfrak{S}}=\|(h_{F})_e\| \leq \|F\|_{\rm{cb}}.
\end{split}\end{equation*}
\end{proof}

\begin{theorem}\label{cnnfdsjfil}
The following are equivalent:

(i) $\mathfrak{B}$ is nuclear;

(ii) $C_{\rm{Red}}^{\ast}(\mathfrak{B})$ is nuclear.

If either of this hold, we have $C^{\ast}(\mathfrak{B})=C^{\ast}_{\rm{Red}}(\mathfrak{B})$.
\end{theorem}
\begin{proof}
By the aid of Proposition \ref{19021402}, we can prove the equivalence of $(i)$ and $(ii)$ by the argument of \cite[Theorem 4.3]{MR3860401}. If $(ii)$ holds, then by \cite[Theorem 25.11]{MR3699795} $C^{\ast}(\mathfrak{B})=C^{\ast}_{\rm{Red}}(\mathfrak{B})$.
\end{proof}

\begin{corollary}\label{vnhudfuinjgfjlg}
If $\mathfrak{B}$ is a Fell bundle over discrete group $G$ such that either $C^{\ast}(\mathfrak{B})$ or $C_{\rm{Red}}^{\ast}(\mathfrak{B})$ is nuclear $C^{\ast}$-algebra, then for any subgroup $H \subset G$, $C^{\ast}(\mathfrak{B}_H)$ and $C_{\rm{Red}}^{\ast}(\mathfrak{B}_H)$ is nuclear $C^{\ast}$-algebra. Furthermore, $\mathfrak{B}_H$ is amenable.
\end{corollary}
\begin{proof}
Notice that if $C^{\ast}(\mathfrak{B})$ is nuclear, then $C_{\rm{Red}}^{\ast}(\mathfrak{B})$ is nuclear because the quotient $C^{\ast}$-subalgebra of a nuclear $C^{\ast}$-algebra is nuclear. 

If $C_{\rm{Red}}^{\ast}(\mathfrak{B})$ is nuclear, then by Theorem \ref{cnnfdsjfil} we have a net $\{\Psi_i\}_{i \in I}$ of Herz-Shur multipliers satisfying $(i)-(iii)$ of Defnition \ref{vnnfduorfjlgf}. It is easy to see that each $\Psi_i|_{\mathfrak{B}_H}$ is completely positive Herz-Schur multiplier of $\mathfrak{B}_H$. Furthermore, since $\{\Psi_i\}_{i \in I}$ satisfies $(i)-(iii)$ of Definition \ref{vnnfduorfjlgf},  $\{\Psi_i|_{\mathfrak{B}_H}\}$ satisfies $(ii)-(iii)$ of Definition \ref{vnnfduorfjlgf}, and by Corollary \ref{19021301} $\{\Psi_i|_{\mathfrak{B}_H}\}$ satisfies $(i)$ of Definition \ref{vnnfduorfjlgf}, now by Theorem \ref {cnnfdsjfil} we conclude that $C_{\rm{Red}}^{\ast}(\mathfrak{B}_H)$ is nuclear. By Theorem \ref {cnnfdsjfil} again, $C_{\rm{Red}}^{\ast}(\mathfrak{B}_H)=C^{\ast}(\mathfrak{B}_H)$, thus $\mathfrak{B}_H$ is amenable. 
\end{proof}

\bigskip

Weijiao He

Queen's University Belfast

email: whe02@qub.ac.uk.

\bibliographystyle{plain}

\bibliography{referencelist}

\end{document}